\documentclass[leqno,12pt]{article}
\usepackage[english]{babel}
\usepackage[T1]{fontenc}
\usepackage[utf8x]{inputenc}

\usepackage{hyperref} 
\usepackage{latexsym,amsmath,amsthm,amssymb,amsfonts,amscd}
\usepackage{graphicx} 
\usepackage{pgf,tikz} 
\usetikzlibrary{shapes} 
\usetikzlibrary{arrows} 
\usepackage{float} 
\evensidemargin\oddsidemargin

\def \e{\mathbb E}
\def \E {{\rm E}} 
\def \p{\mathbb P}
\renewcommand{\P} {{\rm P}}

\def \root{ e }
\def \proot{{e_*}}
\def\px{ {x_*}}

\def\ee{{\rm e}}
\newtheorem{theorembis}{Theorem} 
\def\N{{\mathbb N}}
\newcommand{\fl}[1]{{\lfloor #1 \rfloor}} 
\newcommand{\1}[1]{{{\bf 1}_{\{ {#1} \}}}} 
\newcommand{\parent}[1]{{{#1}_*}} 
\newcommand{\sto}[2][\longrightarrow]{{\substack{ \\ {#1} \\ {#2}}}} 

\def\t{{\mathbb T}}
\def\f{{\mathbb F}}

\def \GW {{\rm GW}}
\def \BW {{\rm BW}}
\def\T{{\mathcal T}}
\def\Tb{{\bf T}}

\def \R{\mathcal R} 
\def \Fb{{\bf F}} 

\def\v{{\mathbb V}}

\def\eps{{\varepsilon}}

\title{Scaling limit of the recurrent biased random walk on a Galton--Watson tree}
\author{Elie A\"id\'ekon\footnote{Sorbonne Universit\'es, UPMC Univ Paris 6, UMR 7599, Laboratoire de Probabilit\'es et Mod\`eles Al\'eatoires, 4 place
Jussieu, F-75005 Paris. E-mail : {\tt Elie.aidekon$@$upmc.fr}} \and  \addtocounter{footnote}{5} Lo\"ic de Raph\'elis \footnote{Sorbonne Universit\'es, UPMC Univ Paris 6, UMR 7599, Laboratoire de Probabilit\'es et Mod\`eles Al\'eatoires, 4 place
Jussieu, F-75005 Paris. E-mail : {\tt Loic.de\_raphelis\_soissan$@$upmc.fr}}}

\begin{document}

\baselineskip=18pt
\setcounter{page}{1}

\renewcommand{\theequation}{\thesection.\arabic{equation}}
\newtheorem{theorem}{Theorem}[section]
\newtheorem{lemma}[theorem]{Lemma}
\newtheorem{proposition}[theorem]{Proposition}
\newtheorem{corollary}[theorem]{Corollary}
\newtheorem{remark}[theorem]{Remark}
\newtheorem{fact}[theorem]{Fact}
\newtheorem{problem}[theorem]{Problem}

\maketitle

\noindent {\bf Abstract}. We show that the trace of the null recurrent biased random walk on a Galton--Watson tree properly renormalized converges to the Brownian forest. Our result extends to the setting of the random walk in random environment on a Galton--Watson tree. \\

\noindent {\bf Keywords}. Random walk, Galton--Watson tree, Scaling limit. \\

\noindent {\bf Mathematics Subject classification (2010)}. 60J80, 60G50, 60F17.

\section{Introduction}

We consider  a Galton--Watson tree $\t$ with offspring distribution $\nu$. The measure $\GW$ denotes the Galton--Watson measure on the space of trees, and $E_{\GW}$ is the expectation with respect to $\GW$. The root is denoted by $\root$. We suppose that the mean number of children $m:=E_{\GW}[\nu]$ is strictly greater than $1$ so that the tree is super-critical. We write $\GW^*$ for the Galton--Watson measure conditioned on $\t$ being infinite. \\ 

We call $\nu(x)$  the number of children of the vertex $x$ in $\t$. For $x\in \t\backslash\{e\}$,  we denote by $\px$ the parent of $x$, that is the neighbour of $x$ which lies on the path from $x$ to the root $\root$, and by $xi,1\le i\le \nu(x)$ the children of $x$. We let $|x|$ be the height of the vertex $x$, that is the graph distance between the root and $x$. Fix a tree $\t$. For $\lambda\ge 0$, the $\lambda$-biased random walk $(X_n)_{n\ge 0}$ is the Markov chain on the graph $\t$ which starts at $\root$ and such that
\begin{eqnarray}
\P_{\t}(X_{n+1}=\px\,|\, X_n=x) &=& {\lambda \over \lambda+\nu(x)}, \label{p(x,px)}\\
\P_{\t}(X_{n+1}=xi\,|\, X_n=x) &=& {1\over \lambda + \nu(x)} \; \; \mbox{for any} \;\; 1\le i\le \nu(x) \label{p(x,xi)}.
\end{eqnarray} 

\noindent To define the transition probabilities from the root, we artificially add a parent $\proot$ to the root, and we suppose that the Markov chain is reflected at $\proot$. We denote by $\P_{\t}$ the quenched probability associated to the Markov chain $(X_n)_n$ on the tree $\t$ and by $\p$, resp. $\p^*$, the annealed probability obtained by averaging $\P_\t$ over $\GW$, resp. $\GW^*$. They are associated to the expectations $\E_{\t}$, $\e$ and $\e^*$.  \\

When $\lambda<m$, the Markov chain is  transient for $\GW^*$-almost every tree, see Lyons~\cite{Lyons}. 
We refer to the works of Lyons, Pemantle and Peres~\cite{LyPePe95}, \cite{LyPePe96}, \cite{LyPePe97} for the study of the transient biased random walk, and open questions. We consider here the null recurrent case $\lambda=m \in(1,\infty)$. Peres and Zeitouni~\cite{PeZe06} showed  a central limit theorem for the height of the walk.\\

{\bf Theorem [Peres, Zeitouni~\cite{PeZe06}]} {\it Assume $m\in(1,\infty)$, $\lambda=m$ and some exponential moments for $\nu$. Let $\sigma^2:= {m(m-1) \over E_{\GW}[\nu(\nu-1)]}$. For $\GW^*$-almost every tree, the process $\left\{|X_{\lfloor nt \rfloor }| / \sqrt{  \sigma^2 n}\right\}_{t\ge 0} $ converges in law towards $(|B_{t}|)_{t\ge 0}$, where $(B_{t})_{t\ge 0}$ is a standard Brownian motion.}\\
 
This theorem was proved by finding an explicit invariant measure on the space of trees, and showing an invariance principle for a martingale which approximates the process $(|X_{n}|)_{n\ge 0}$. Dembo and Sun~\cite{dembo} extended the theorem to the case where $\t$ is a multi-type Galton--Watson tree, assuming only a moment of order $4+\varepsilon$, for some $\varepsilon>0$. A natural  question is now to understand the trace of the walk $(X_{n})_{n}$ in the tree $\t$. Let $\R_{n}:=\{X_{k},\, k\le n\}$ be the set of vertices visited by the walk until time $n$, and $R_{n}$ its cardinal, also called the range. Notice that $\R_{n}$ is a tree. We will consider it as a metric space, where each edge has length $1$.  From this point of view, $\R_{n}$ is an unlabeled tree. Our main theorem says that $\R_{n}$ suitably normalized converges in distribution to  the real tree coded by $(|B_{t}|)_{t\in[0,1]}$. We briefly recall its construction taken from~\cite{DuLG05}. We refer to~\cite{legall} for a review on random real trees. \\

Let $g$ be a continuous function from $[0,1]$ to $\mathbb{R}_{+}$ (it is usually assumed that $g(1)=0$ but it is not the case here). For any $s,t \in [0,1]$, define $d_{g}(s,t):= g(s)+g(t)-2 \min\{g(r); r\in[\min(s,t),\max(s,t)]\}$. Using the equivalence relation $s\sim t \Leftrightarrow d_{g}(s,t)=0$, we see that $d_{g}$  defines a metric on the quotient space $\T_{g}:= [0,1] / \sim$. The metric space $(\T_{g},d_{g})$ is the real tree encoded by $g$. The space of all real trees is equipped with the Gromov-Hausdorff metric $d_{G}$ (see Section 1 of~\cite{legall}). Taking  for $g$ a normalized Brownian excursion, $\T_{g}$ is the continuum random tree, also called Brownian tree, introduced by Aldous~\cite{aldousI},\cite{aldousIII}. Here we will take for $g$ the reflected Brownian motion $|{\bf B}|:=(|B_{t}|)_{t\in[0,1]}$. In this case, $\T_{g}$ can be seen as a Brownian forest explored up to time $1$.  For $r>0$, the notation $r\R_{n}$ denotes the tree $\R_{n}$ with edge length $r$. \\

\begin{theorem}\label{t:main}
Assume that $m\in(1,\infty)$, $\lambda=m$ and let $\sigma^2:= {m(m-1) \over E_{\GW}[\nu(\nu-1)]}$. Under $\p^*$ (annealed case) and under $\P_{\t}$ for $\GW^*$-a.e. tree $\t$ (quenched case), the following joint  convergence in law holds as $n\to\infty$:
$$
{ 1\over \sqrt{  \sigma^2 n}} \left(\left\{   |X_{\lfloor nt \rfloor }| \right\}_{t\in [0,1]},\R_{n}\right) \Rightarrow \left(|{\bf B} |,\T_{|{\bf B}| } \right)
$$
for the  Skorokhod topology on the space of c\`adl\`ag functions and the Gromov-Hausdorff topology on the space of real trees.
\end{theorem}

Therefore, asymptotically, the random walk $(X_{n})_{n}$ looks like the contour function of its trace, see Section 3 of~\cite{legall} (this statement is not very precise because we deal with unlabeled trees here).  The theorem also extends the result of Peres and Zeitouni~\cite{PeZe06} on the convergence of the height of the random walk under a second moment assumption. The idea of the proof is to look at the local times of the random walk. This strategy was used in the papers of Kesten, Kozlov, Spitzer~\cite{kks} in the case of random walks in random environment on $\mathbb{Z}$, and Basdevant and Singh~\cite{arvinda},\cite{arvindb},\cite{arvindc} in the case of multi-excited random walks on $\mathbb{Z}$ and on the regular tree.   Call an excursion of $(X_{n})_{n}$ the trajectory of the walk before hitting the parent of its starting point. During one excursion, the local times of the edges $(\px,x)$ (i.e. the number of times the directed edge has been crossed) form under $\p$ a multi-type Galton--Watson tree, with initial type $1$. We will show that the successive excursions from the root $\root$ are close to be independent (they are identically distributed but not independent under $\p$). Therefore, $\R_{n}$ is close to a concatenation of i.i.d.\ multi-type Galton--Watson trees, and we use a result of~\cite{loic} on scaling limit of multi-type Galton--Watson trees to complete the proof in the annealed case. In the quenched case, we show that there is an averaging phenomenon, which is reminiscent of (but much easier to prove than) what happens for the large deviations of transient biased random walks on Galton--Watson trees~\cite{dgpz02}.

\bigskip

The same strategy can be applied to the case of random walks on random environment on Galton--Watson trees, see Faraud~\cite{faraud}. We give more details on this account in Section~\ref{s:rwre}.

\bigskip

The paper is organized as follows. In Section~\ref{s:preliminaries}, we recall the result of~\cite{loic} on the scaling limit of certain two-type Galton--Watson trees with edge lengths. In Section~\ref{s:multi-type}, we describe the process of the local times of an excursion of the biased random walk. In Section~\ref{s:reduction}, we construct different reduced trees associated to the trace of an excursion. These reduced trees are simpler to deal with since they fall into the scope of~\cite{loic}, but still contain all the information needed.  We prove Theorem~\ref{t:main} in Section~\ref{s:proof}.  Finally, Section~\ref{s:rwre} deals with the case of random walks in random environment on a Galton--Watson tree.

\section{Preliminaries}
\label{s:preliminaries}
Let $T$ be a finite rooted ordered tree. We refer to Neveu~\cite{Neveu} for the formal construction of a tree. By the representation of~\cite{Neveu}, we can label the vertices of a tree through the set of words $\bigcup_{n\ge 0} \mathbb{N}^n$.  The generation of a vertex is the length of its label, the root being of generation $0$. Since the set of words is equipped with the lexicographical order, we can rank the vertices of $T$ from the smallest (the root) to the biggest. This gives a way to explore the tree, starting from the root and going clockwise, also called depth-first search. The index of a vertex is the rank of the vertex in the depth-first search, the index of the root being set to $0$. \\

 We put on each edge of the tree a non-negative mark, which stands for its length. When not specified, the length of an edge is set to $1$. Therefore the tree $T$ is endowed with a natural metric (or pseudo-metric in the case where some edges have length $0$). The height of a vertex is by definition the distance of the vertex to the root. The height function of the tree $T$ is the function that maps any integer $k\in [\![0,\#T-1 ]\!]$ to the height of the vertex of index $k$ in the depth-first search.  A forest is a sequence of finite  trees $(T_{i})_{i}$, and the height function of a forest is the concatenation of the height functions of the trees $(T_{i})_{i}$. \\

Let us introduce the result of~\cite{loic} that will be used in our proof. Let $T$ be a leafed Galton--Watson tree with edge lengths (as defined in Section~1.1 of \cite{loic}) ; that is $T$ is a 2-type Galton--Watson tree with edge lengths with types denoted by $s$ and $f$ such that vertices of type $s$ give no offspring (they are sterile). The offspring distribution of a vertex of type $f$ can be represented by a random point process $\theta=(\delta_{t(i),\ell(i)})_{i\le N}$ on $\{s,f\}\times \mathbb R_{+}$  where $N$ is the number of children, $t(i)$ is the type ($s$ or $f$) of the $i$-th child for the lexicographical order, and $\ell(i)$ is the length of its edge.  The type of the root is $f$. It gives birth according to $\theta$. Children at generation $1$ of type $s$ have no offspring, while the ones of type $f$  give birth independently according to i.i.d.\ copies of $\theta$ and so on. We suppose that $\sum_{i=1}^{N} {\bf 1}_{\{ t(i)=f \}}$ has mean $1$ and some finite variance $\Sigma_{f}^2\in (0,\infty)$ (we take $\Sigma_{f}\ge 0$). It means that the Galton--Watson tree composed of vertices of type $f$ is critical. Moreover, we suppose that \\
(i) $E[N]=: C_{1}^{-1}<\infty$,\\
(ii) $y^{2}E\left[ \sum_{i=1}^N {\bf 1}_{\{ t(i)=f, \ell(i)>y\}} \right]$ goes to $0$ as $y\to\infty$, \\
(iii) $y^2P\left(\max_{i\le N,t(i)=s}\ell(i)>y\right)$ goes to $0$ as $y\to\infty$. \\

\noindent We denote by $C_2$ the quantity $E\left[\sum_{i=1}^N \ell(i){\bf 1}_{\{ t(i)=f \}}\right]$, which is finite thanks to (ii). Finally we take i.i.d.\ trees $(T_{i})_{i}$ distributed as $T$ and we call $H$ the height function associated to the forest, and $H_{f}$ the height function of the forest restricted to vertices of type $f$.  The following result comes from Theorem~1 of~\cite{loic} (beware that $H_f$ is different from the process $H^1$ introduced in~\cite{loic} since in our case the lengths are not reset to $1$). It states that the height function $H$ is asymptotically given by a deterministic rescaling in time of $H_{f}$. Loosely speaking the forest composed of the vertices of type $f$ captures all the randomness. \\

\begin{theorembis}\label{t:loic} {\bf \cite{loic}} The following joint convergence in law holds as $n\to \infty$ for the Skorokhod topology of c\`adl\`ag functions~:
$$
{1\over \sqrt{n}}\left(H(\lfloor nt\rfloor),H_{f}(\lfloor nt\rfloor)\right)_{t\ge 0} \Rightarrow {2C_{2}\over \Sigma_{f}} \left(|B_{C_{1} t }|,|B_{t}|\right)_{t\ge 0}, 
$$
where $B$ is a standard Brownian motion. 
\end{theorembis}

To be precise, you obtain this result by combining Theorem 1 of~\cite{loic} applied on one hand to the leafed Galton--Watson tree with edge lengths $T$ and on the other hand to the single-type Galton--Watson tree with edge lengths composed only of vertices of type $f$. \\

Let $u(k)$ be the index (for the first-depth search) of the $k$-th vertex of type $f$ visited by  the first-depth search in the forest. Let $ \bar \ell(n)$ be the maximal length of the edges explored by the first-depth search until $n$ vertices of type $f$ have been visited. The following lemma can be found in~\cite{loic} ((i) is the equation which lies right below equation (2.10) in the proof of Proposition 5, and (ii) comes from equation (2.4) in the proof of Proposition~5 together with the use of our condition (iii) to control the length of edges of type $s$) .
\begin{lemma} \label{l:loic}The following convergences hold in probability:
$$
(i)\; \lim_{k\to\infty} {u(k) \over k} \stackrel{(\P)}{=}   {1\over C_{1}} ; \, \;\;\;\;\;
(ii)\; \lim_{n\to\infty} {{ \bar \ell}(n)  \over \sqrt{n} } \stackrel{(\P)}{=}  0.
$$
\end{lemma}

\section{Description of the process of local times}

\label{s:multi-type}

Recall that $\t$ is a Galton--Watson tree, in which we artificially added a parent $\proot$ to the root $\root$. Let $\tau_{\proot}^{(1)}:=\min\{n\ge 1\,:\, X_{n}=\proot\}$ be the hitting time of $\proot$ in $\t$. Let $N_{\root}^{(1)}:=1$ and for each vertex  $x\notin\{\proot,\root\}$,
\begin{equation*}\label{def:Nx}
N_{x}^{(1)}:=\sum_{n= 1}^{\tau_{\proot}^{(1)}} {\bf 1}_{\{X_{n-1}=\px,X_{n}=x\}}
\end{equation*}
which stands for the number of crosses of the directed edge $(\px,x)$ during one excursion. More generally, we define for $k\ge 2$, 
$$
\tau_{\proot}^{(k)} := \min\{n>\tau_{\proot}^{(k-1)}\,:\, X_{n} =\proot\},
$$

\noindent then $N_{\root}^{(k)}:=k$ and for each vertex $x\notin\{\proot,\root\}$,
$$
N_{x}^{(k)} := \sum_{n= 1}^{\tau_{\proot}^{(k)}} {\bf 1}_{\{X_{n-1}=\px,X_{n}=x\}}.
$$

\noindent The random variable $\tau_{\proot}^{(k)}$ stands for the $k$-th visit time at $\proot$, and $N_{x}^{(k)}$ is the local time on the directed edge $(\px,x)$ up to that time. 
\begin{lemma}\label{l:multi-type}
Let $k\ge 1$. Under $\p$, the tree $\{x\in \t\backslash\{\proot\}\,:\, N_{x}^{(k)}\ge 1\}$ is a multi-type Galton-Watson tree with initial type $k$.
\end{lemma}

\noindent {\it Proof}. Let us construct the tree $\t$ and the Markov chain $(X_{n})_{n\ge 0}$. Recall from the setting of Neveu~\cite{Neveu}, that we can see $\t\backslash\{\proot\}$ as a subset of the set of words $U:=\bigcup_{n\ge 0} \mathbb{N}^n$.  On each word $x\in U$ independently, we let $\nu(x)$ be distributed as $\nu$. In the case where $x$ is a vertex of $\t$, $\nu(x)$ is the  number of children of $x$ in $\t$. Furthermore,  on each word $x\in U$ independently, we attach a sequence $\mathcal P_x$  of i.i.d.\ random variables equal to  a child $xi$ (with $i\le \nu(x)$) with probability ${1\over m+\nu(x)}$ and to the parent $\px$ with probability ${m\over m+\nu(x)}$. Then the Markov chain $(X_n)_{n\ge 0}$ is a function of all processes $(\mathcal P_x,x\in U)$.  For a child $xi$ of $x$, we observe that $N_{xi}^{(k)}$ is  the  number of appearances of $xi$  in the sequence $\mathcal P_x$ until $\px$ has appeared $N_{x}^{(k)}$ times. In particular, the law of $(N_{xi}^{(k)},i\le \nu(x))$ given $(N_{y}^{(k)},|y|\le |x|)$ only depends on $N_{x}^{(k)}$. It implies the lemma. $\Box$ \\

Let us consider the setting of \cite{loic}. The mean matrix  is, for $i,j\ge1$,
$$
m_{i,j}:=\e\left[\sum_{|x|=1} {\bf 1}_{\{N_{x}^{(k)}=j\}} \mid N_{\root}^{(k)}=i\right]  =  \binom{i+j-1}{j} {m^{i+1}\over (m+1)^{i+j}}.
$$

\noindent  We notice that the vectors $(a_{i})_{i\ge 1}$ and $(b_{i})_{i\ge 1}$ given by $a_{i}:=(m-1)m^{-i}$ and $b_{i}:=(1-m^{-1})\, i$ are respectively left and right eigenvectors associated to the eigenvalue $1$, normalized such that $\sum_{i\ge 1} a_{i}=1$ and $\sum_{i\ge 1} a_{i}b_{i}=1$. In this context, a version of the many-to-one lemma reads as follows (its proof goes by induction on $n$). For any bounded function $f:\mathbb{N}^n\to\mathbb{R}$, we have, denoting by $x_{i}$ the ancestor of $x$ at generation $i$, 
\begin{equation}\label{eq:many-to-one}
\e\left[\sum_{|x|=n,N_{x}^{(k)}\ge 1} f(N_{x_{1}}^{(k)},N_{x_{2}}^{(k)},\ldots,N_{x_{n-1}}^{(k)},N_{x}^{(k)})  \right] 
=
k\e\left[{1\over \hat N_{n}}f(\hat N_{1},\hat N_2,\ldots,\hat N_{n-1},\hat N_{n}) \Big| \hat N_{0} = k\right]
\end{equation}

\noindent where $(\hat N_{i})_{i\ge 0}$ is a Markov chain on $\mathbb{N}\backslash\{0\}$  with transition probabilities from $i$ to $j$ given by
$$
m_{i,j}{b_{j}\over b_{i}} = \binom{i+j-1}{i} {m^{i+1}\over (m+1)^{i+j}}.
$$

\noindent We can check that the probability distribution $\pi$ on $\mathbb{N}\backslash\{0\}$  given by $\pi_{i}:=a_{i}b_{i}$ is then a reversible measure for $(\hat N_{i})_{i\ge 0}$.  The return time at $1$ of this Markov chain is easily controlled by the following lemma. 
\begin{lemma}\label{l:lyapounov}
Let $\hat \gamma_{1}:=\min\{i\ge 1\,:\, \hat N_{i} =1\}$. There exists $r>0$ such that $\e\left[{\rm e}^{r\hat \gamma_{1}}\, |\, \hat N_{0}=1\right]<\infty$.
\end{lemma}
\noindent {\it Proof}.  A computation leads to 
\begin{equation*}
\sum_{j≥1}\hat{p}_{i,j} j =1+\frac{1}{m}(i+1).
\end{equation*}
Now for all $i>i_0$ large enough, $1+\frac{1}{m}(i+1)<d\times i$ for some $d<1$. It implies that, starting in the set $\{i\le i_{0}\}$, the return time to this set  admits exponential moments (see e.g. Theorem 15.2.5 in~\cite{meyn}). The probability to go from $i\le i_{0}$ to $1$ in one step is uniformly bounded from below by some positive constant. It implies that the number of hits of the set $\{i\le i_{0}\}$ before time $\hat \gamma_{1}$ is stochastically dominated by a geometric random variable.Therefore $\hat \gamma_{1}$ is stochastically dominated by a sum of a geometric number of i.i.d random variables which have exponential moments. It implies the lemma. $\Box$\\        
 
 \section{Reduction of trees}
 \label{s:reduction}

Let 
$$
\Tb:= \{x\in \t\backslash\{\proot\}  \,:\, N_{x}^{(1)}\ge 1\}.
$$

By Lemma~\ref{l:multi-type}, we know that $\Tb$ is a multi-type Galton--Watson tree. Following an idea of Miermont~\cite{miermont} further developed in~\cite{loic}, we will see that the important vertices are the vertices of type $1$. Therefore, we choose to work with some simpler trees constructed as follows.\\

\noindent {\it The tree $\Tb^{(r)}$}\\
Draw the tree $\Tb$ in the plane and erase all the edges (but keep the vertices, remember the genealogy and the trajectory $(X_{n})_{n\le \tau_{\proot}^{(1)}}$). Draw an edge between $x$ and any descendant $y$ such that $x$ is the youngest ancestor of $y$ in $\Tb$ with type $1$ (excluding $y$ itself). The length of the edge between $x$ and $y$ is set to be $|y|-|x|$, where $|z|$ is the generation of $z$ in the tree $\t$ or equivalenty in $\Tb$. Re-order the resulting tree  so that the order in which the walk $(X_{n})_{n}$ first hits the vertices is given by the depth-first search order (notice that it is possible indeed using the fact that an edge $(\px,x)$ is crossed only once upwards for a vertex $x\in\t$ with type $1$). Call $\Tb^{(r)}$ the tree that you obtain, see Figure~\ref{f:Tr}. The tree $\Tb^{(r)}$ will be used to encode the trace $\R_{n}$.  This reduced tree is studied in Sections~1.3 and~3 of~\cite{loic}.  From Proposition~1 of~\cite{loic}, we see that $\Tb^{(r)}$ is a leafed Galton--Watson tree with edge lengths as introduced in Section~\ref{s:preliminaries}. We set the type of a vertex $z$ as $f$ if $N_{z}^{(1)}=1$ and as $s$ otherwise.  Conditions (ii) and (iii) are satisfied (see Appendix of~\cite{loic}, equation (A.1) there is satisfied with $V(i)=i$ as shown in our proof of Lemma~\ref{l:lyapounov}). The constants $\Sigma_{f}$, $C_{1}$ and $C_{2}$ are computed in Section~3.4 of~\cite{loic}. We have $\Sigma_{f} = {\eta \over b_{1}\sqrt{a_{1}}}$, $C_{1}=a_{1}$ and $C_{2}={1\over a_{1}b_{1}}$, where $a_{1},b_{1}$ are defined in Section~\ref{s:multi-type} and $\eta>0$ is given by $\eta^2=2\frac{E_{\GW}[\nu(\nu-1)]}{m^2}<\infty$ (in the setting and notation of~\cite{loic}, we have for all $i,j,k\geq 1$, $Q_{j,k}^{i}=\binom{i+j+k-1}{i-1,j,k}E_{\GW}[\nu(\nu-1)]\frac{m^i}{(m+2)^{i+j+k}}$).

\begin{figure}
\begin{tikzpicture}[line cap=round,line join=round,>=triangle 45,x=0.8cm,y=1.0cm]
\clip(-1.5559804878048776,-1.382497560975608) rectangle (19.316468292682927,9.901658536585366);
\draw (8.,1.)-- (4.,2.);
\draw (8.,1.)-- (12.,2.);
\draw (4.,2.)-- (1.,3.);
\draw (4.,2.)-- (7.,3.);
\draw (12.,2.)-- (9.,3.);
\draw (12.,2.)-- (14.,3.);
\draw (1.,3.)-- (-1.,4.);
\draw (1.,3.)-- (3.,4.);
\draw (-1.,4.)-- (-1.,5.);
\draw (3.,4.)-- (3.,5.);
\draw (9.,3.)-- (9.,4.);
\draw (9.,4.)-- (9.,5.);
\draw (14.,3.)-- (12.,4.);
\draw (14.,3.)-- (16.,4.);
\draw (3.,5.)-- (3.,6.);
\draw (14.,3.)-- (14.,4.);
\draw (14.,4.)-- (14.,5.);
\draw (3.,6.)-- (2.,7.);
\draw (3.,6.)-- (4.,7.);
\draw (14.,5.)-- (13.,6.);
\draw (14.,5.)-- (15.,6.);
\draw (8.,1.)-- (8.,0.);
\draw (8.071141463414635,1.040546341463416) node[anchor=north west] {$\color{red}{1}, \color{green}{0} $};
\draw (3.708058536585366,1.893697560975611) node[anchor=north west] {$\color{red}{2},\color{green}{1}$};
\draw (6.616780487804879,2.808497560975611) node[anchor=north west] {$\color{red}{3},\color{green}{2}$};
\draw (11.333014634146341,1.8975804878048794) node[anchor=north west] {$\color{red}{2},\color{green}{3}$};
\draw (13.50679024390244,2.832526829268294) node[anchor=north west] {$\color{red}{1},\color{green}{4}$};
\draw (11.473307317073171,4.758243902439025) node[anchor=north west] {$\color{red}{2},\color{green}{5}$};
\draw (15.480565853658537,4.758243902439025) node[anchor=north west] {$\color{red}{5},\color{green}{6}$};
\draw (14.07038048780488,4.227180487804879) node[anchor=north west] {$\color{red}{2},\color{green}{7}$};
\draw (14.07038048780488,5.240039024390245) node[anchor=north west] {$\color{red}{1},\color{green}{8}$};
\draw (12.390048780487805,6.732019512195123) node[anchor=north west] {$\color{red}{1},\color{green}{9}$};
\draw (14.341736585365853,6.732019512195123) node[anchor=north west] {$\color{red}{2},\color{green}{10}$};
\draw (8.420409756097562,2.806556097560977) node[anchor=north west] {$\color{red}{1},\color{green}{11}$};
\draw (9.058029268292684,4.227180487804879) node[anchor=north west] {$\color{red}{3},\color{green}{12}$};
\draw (8.64249756097561,5.99319024390244) node[anchor=north west] {$\color{red}{6},\color{green}{13}$};
\draw (0.5136585365853662,2.804468292682928) node[anchor=north west] {$\color{red}{1},\color{green}{14}$};
\draw (-1.7380292682926826,3.871356097560977) node[anchor=north west] {$\color{red}{1},\color{green}{15}$};
\draw (-1.7601170731707314,5.719160975609757) node[anchor=north west] {$\color{red}{2},\color{green}{16}$};
\draw (2.6692292682926833,3.9934439024390254) node[anchor=north west] {$\color{red}{5},\color{green}{17}$};
\draw (3.0847609756097567,5.240039024390245) node[anchor=north west] {$\color{red}{1},\color{green}{18}$};
\draw (3.0847609756097567,6.226926829268294) node[anchor=north west] {$\color{red}{2},\color{green}{19}$};
\draw (1.300546341463415,7.744878048780488) node[anchor=north west] {$\color{red}{4},\color{green}{20}$};
\draw (3.5522341463414637,7.748907317073172) node[anchor=north west] {$\color{red}{2},\color{green}{21}$};
\draw (7.4,-0.)  node[anchor=north west]{$\proot$};
\draw (7.4,1.)  node[anchor=north west]{$\root$};
\begin{scriptsize}
\draw [fill=black] (7.9,0.9) rectangle (8.1,1.1);
\draw [fill=black] (4.,2.) circle (2pt);
\draw [fill=black] (12.,2.) circle (2pt);
\draw [fill=black] (0.9,2.9) rectangle (1.1,3.1);
\draw [fill=black] (7.,3.) circle (2pt);
\draw [fill=black] (8.9,2.9) rectangle (9.1,3.1);
\draw [fill=black] (13.9,2.9) rectangle (14.1,3.1);
\draw [fill=black] (-1.1,3.9) rectangle (-0.9,4.1);
\draw [fill=black] (3.,4.) circle (2pt);
\draw [fill=black] (-1.,5.) circle (2pt);
\draw [fill=black] (2.9,4.9) rectangle (3.1,5.1);
\draw [fill=black] (4.,7.) circle (2pt);
\draw [fill=black] (9.,4.) circle (2pt);
\draw [fill=black] (9.,5.) circle (2pt);
\draw [fill=black] (14.,3.) circle (2pt);
\draw [fill=black] (12.,4.) circle (2pt);
\draw [fill=black] (16.,4.) circle (2pt);
\draw [fill=black] (3.,6.) circle (2pt);
\draw [fill=black] (14.,4.) circle (2pt);
\draw [fill=black] (13.9,4.9) rectangle (14.1,5.1);
\draw [fill=black] (2.,7.) circle (2pt);
\draw [fill=black] (12.9,5.9) rectangle (13.1,6.1);
\draw [fill=black] (15.,6.) circle (2pt);
\draw [fill=black] (8.,0.) circle (0.5pt);
\end{scriptsize}
\draw [fill=black] (16.,1.) circle (2pt);
\draw (16.4,1.) node[right]{type $\neq 1$};
\draw [fill=black] (15.9,1.4) rectangle (16.1,1.6);
\draw (16.4,1.5) node[right] {type $1$};
\end{tikzpicture}

\definecolor{sqsqsq}{rgb}{0.12549019607843137,0.12549019607843137,0.12549019607843137}
\begin{tikzpicture}[line cap=round,line join=round,>=triangle 45,x=0.8cm,y=1.0cm]
\clip(-1.5,-0.8955512572533862) rectangle (18.65421,8.);
\draw[|->] (-1,1)-- (-1,6.7) ;
\draw (-1,7) node[above]{\hspace{1.1cm}\small Height of \phantom{))}};
\draw (-0.85,6.7) node[above]{\hspace{1.1cm}\small the vertex \phantom{))}};
\foreach \y in {1.0,2.0,3.0,4.0,5.0,6.0}
\draw[shift={(-1,\y)},color=black] (2pt,0pt) -- (-2pt,0pt);
\draw (6.,0.)-- (6.,1.);
\draw (6.,1.)-- (0.,2.);
\draw (6.,1.)-- (1.,3.);
\draw (6.,1.)-- (5.,2.);
\draw (6.,1.)-- (6.,3.);
\draw (6.,1.)-- (10.,3.);
\draw (6.,1.)-- (13.,3.);
\draw (-0.30058018939393987,2.7071566731141192) node[anchor=north west] {$\color{red}{2},\color{green}{1}$};
\draw (0.2586402777777769,3.6690522243713726) node[anchor=north west] {$\color{red}{3},\color{green}{2}$};
\draw (4.358464722222219,2.6812379110251444) node[anchor=north west] {$\color{red}{2},\color{green}{3}$};
\draw (5.993171654040401,2.92201160541586) node[anchor=north west] {$\color{red}{1},\color{green}{4}$};
\draw (9.656712260101005,2.9038684719535778) node[anchor=north west] {$\color{red}{1},\color{green}{11}$};
\draw (12.656530757575752,2.9257253384912955) node[anchor=north west] {$\color{red}{1},\color{green}{14}$};
\draw (6.,3.)-- (4.,4.);
\draw (6.,3.)-- (5.,4.);
\draw (6.,3.)-- (6.,4.);
\draw (6.,3.)-- (7.,5.);
\draw (7.,5.)-- (6.,6.);
\draw (7.,5.)-- (8.,6.);
\draw (10.,3.)-- (9.,4.);
\draw (10.,3.)-- (10.,5.);
\draw (13.,3.)-- (12.,4.);
\draw (13.,3.)-- (14.,4.);
\draw (13.,3.)-- (16.,5.);
\draw (16.,5.)-- (15.,6.);
\draw (16.,5.)-- (16.,7.);
\draw (16.,5.)-- (17.,7.);
\draw (3.0155161994949474,4.723558994197292) node[anchor=north west] {$\color{red}{2},\color{green}{5}$};
\draw (4.384736666666664,4.727272727272727) node[anchor=north west] {$\color{red}{5},\color{green}{6}$};
\draw (5.625721489898987,4.7218762088974854) node[anchor=north west] {$\color{red}{2},\color{green}{7}$};
\draw (7.040434368686864,5.197292069632495) node[anchor=north west] {$\color{red}{1},\color{green}{8}$};
\draw (5.4323861742424215,6.702901353965184) node[anchor=north west] {$\color{red}{1},\color{green}{9}$};
\draw (7.588083876262622,6.7033655705996132) node[anchor=north west] {$\color{red}{2},\color{green}{10}$};
\draw (8.34004757575757,4.6787234042553186) node[anchor=north west] {$\color{red}{3},\color{green}{12}$};
\draw (12.,4.)-- (12.,5.);
\draw (9.257116893939389,5.6785029013539652) node[anchor=north west] {$\color{red}{6},\color{green}{13}$};
\draw (12.069279936868681,4.396518375241779) node[anchor=north west] {$\color{red}{1},\color{green}{15}$};
\draw (11.373787525252519,5.693230174081238) node[anchor=north west] {$\color{red}{2},\color{green}{16}$};
\draw (13.321050239898983,4.678742746615086) node[anchor=north west] {$\color{red}{5},\color{green}{17}$};
\draw (16.09851229797979,4.918762088974854) node[anchor=north west] {$\color{red}{1},\color{green}{18}$};
\draw (13.481642323232316,6.620309477756286) node[anchor=north west] {$\color{red}{2},\color{green}{19}$};
\draw (15.242041010101002,7.726615087040619) node[anchor=north west] {$\color{red}{4},\color{green}{20}$};
\draw (17.01126147727272,7.626963249516441) node[anchor=north west] {$\color{red}{2},\color{green}{21}$};
\draw (6.101126147727272,0.9626963249516441) node[anchor=north west] {$\color{red}{1},\color{green}{0}$};
\draw (5.4,-0.)  node[anchor=north west]{$\proot$};
\draw (5.4,1.)  node[anchor=north west]{$\root$};
\begin{scriptsize}
\draw [fill=sqsqsq] (6.,0.) circle (0.5pt);
\draw [fill=sqsqsq] (5.9,0.9) rectangle (6.1,1.1);
\draw [fill=sqsqsq] (0.,2.) circle (2pt);
\draw [fill=sqsqsq] (1.,3.) circle (2pt);
\draw [fill=sqsqsq] (5.,2.) circle (2pt);
\draw [fill=sqsqsq] (5.9,2.9) rectangle (6.1,3.1);
\draw [fill=sqsqsq] (9.9,2.9) rectangle (10.1,3.1);
\draw [fill=sqsqsq] (12.9,2.9) rectangle (13.1,3.1);
\draw [fill=sqsqsq] (4.,4.) circle (2pt);
\draw [fill=sqsqsq] (5.,4.) circle (2pt);
\draw [fill=sqsqsq] (6.,4.) circle (2pt);
\draw [fill=sqsqsq] (6.9,4.9) rectangle (7.1,5.1);
\draw [fill=sqsqsq] (5.9,5.9) rectangle (6.1,6.1);
\draw [fill=sqsqsq] (8.,6.) circle (2pt);
\draw [fill=sqsqsq] (9.,4.) circle (2pt);
\draw [fill=sqsqsq] (10.,5.) circle (2pt);7
\draw [fill=sqsqsq] (11.9,3.9) rectangle (12.1,4.1);
\draw [fill=sqsqsq] (14.,4.) circle (2pt);
\draw [fill=sqsqsq] (15.9,4.9) rectangle (16.1,5.1);
\draw [fill=sqsqsq] (15.,6.) circle (2pt);
\draw [fill=sqsqsq] (16.,7.) circle (2pt);
\draw [fill=sqsqsq] (17.,7.) circle (2pt);
\draw [fill=sqsqsq] (12.,5.) circle (2pt);
\end{scriptsize}
\draw [fill=sqsqsq] (16.,1.) circle (2pt);
\draw (16.4,1.) node[right]{type s};
\draw [fill=sqsqsq] (15.9,1.4) rectangle (16.1,1.6);
\draw (16.4,1.5) node[right] {type f};
\end{tikzpicture}
\caption{The tree $\Tb$ (top) and the tree $\Tb^{(r)}$ (bottom) associated.  $(\color{red}{2},\color{green}{5})$ means that the vertex has type $N_{x}^{(1)}=2$ and is the $5$-th distinct vertex visited by the walk.}
\label{f:Tr}
\end{figure}
\bigskip

\noindent {\it The tree $\Tb^{(w)}$}\\
Consider the tree $\Tb^{(r)}$,  and let $x$ be a vertex of type $f$. For any child $y$ of $x$ in $\Tb^{(r)}$ with type $s$, duplicate $k_{y}-1$ times the edge $(x,y)$ (and root the duplicated edges at $x$) where we denote by $k_{y}$ the number of times the vertex $y$ has been visited by the walk $(X_{n})_{n\le \tau_{\proot}^{(1)}}$ in $\t$. Root also $k_{x}-1$ edges of length $0$ at $x$. Do this for any vertex $x$ of type $f$ and re-order the new edges so that the height function of the tree is exactly given by $(|X_{n}|)_{n\le \tau_{\proot}^{(1)}}$. Call $\Tb^{(w)}$ this tree. Again $\Tb^{(w)}$ is a leafed Galton--Watson tree with edge lengths as introduced in Section~\ref{s:preliminaries}. The old vertices inherit their types $s$ or $f$ from  $\Tb^{(r)}$ whereas the type of the newly created vertices are all set to $s$.  The values of $\Sigma_{f}$ and $C_{2}$ remain unchanged but we need to compute $C_{1}$. We observe that in that case $C_{1}^{-1}$ is by construction 
$$
 \e\left[ k_{\root} -1 + \sum_{x\in \t\backslash\{\proot,\root\}} k_x{\bf 1}_{\{ N_{x_{i}}^{(1)}\neq 1,\forall i \in [\![ 1, |x|]\!] \}} + \sum_{x\in \t\backslash\{\proot,\root\}} {\bf 1}_{\{ N_{x_{i}}^{(1)}\neq 1,\forall i \in [\![ 1, |x|-1]\!] \}}{\bf 1}_{\{ N_{x}^{(1)}=1\}} \right]
$$

\noindent  where we recall that $x_{i}$ is the ancestor of $x$ at generation $i$ in $\t$. We notice that the term inside the expectation can be rewritten as
$$
2 \sum_{  x\in \t\backslash\{\proot,\root\}} N_{x}^{(1)}{\bf 1}_{\{ N_{x_{i}}^{(1)}\neq 1,\forall i \in [\![ 1, |x|-1]\!] \}}.
$$

\noindent Therefore 
$$
C_{1}^{-1}
=
2 \sum_{\ell\ge 1} \e\left[ \sum_{|x|=\ell} N_{x}^{(1)}{\bf 1}_{\{ N_{x_{i}}^{(1)}\neq 1,\forall i \in [\![ 1, \ell-1]\!] \}}  \right].
$$

\noindent By equation (\ref{eq:many-to-one}), we get 
\begin{eqnarray*}
2\e\left[ \sum_{x\in \t\backslash\{\proot,\root\}} N_{x}^{(1)}{\bf 1}_{\{ N_{x_{i}}^{(1)}\neq 1,\forall i \in [\![ 1, |x|-1]\!] \}}  \right]
&=&
2\sum_{\ell\ge 1} \p\left(\hat{N}_{i}^{(1)}\neq 1,\forall i \in [\![ 1, \ell-1]\!]  \,|\, \hat N_{0} =1  \right)\\ 
&=& 2 \e\left[  \hat \gamma_{1}  \,|\, \hat N_{0} =1  \right] 
=
{2\over \pi(1)}=\frac{2}{a_1b_1}
\end{eqnarray*}

\noindent where we recall that $\pi$ is the invariant probability measure of the Markov chain $(\hat N_{k})_{k}$. Therefore the value of $C_{1}$ is now $\frac{a_1b_1}{2}$. \\

{\bf Remark}. During the procedure, and by an abuse of notation,  a vertex $x\in\t$ was referred to by the same name in the trees $\Tb$, $\Tb^{(r)}$ and $\Tb^{(w)}$. We will always do so. \\

Let $(\Tb_i,\Tb_{i}^{(w)},\Tb_{i}^{(r)})_{i\ge 1}$ be i.i.d.\ copies of the trees $(\Tb,\Tb^{(w)},\Tb^{(r)})$. We call $\Fb$, $\Fb^{(w)}$, $\Fb^{(r)}$ the forests associated. We let $H^{(w)}$ and $H^{(r)}$  be the height functions of $\Fb^{(w)}$ and $\Fb^{(r)}$. Theorem A yields the following proposition. 

\begin{proposition}\label{p:loic}
 The following joint convergence in law holds as $n\to\infty$ for the Skorokhod topology on the space of c\`adl\`ag functions: 
 $$
 {1\over \sqrt{\sigma^2 n}}\left(  H^{(w)}( {\lfloor nt \rfloor}), H^{(r)}({\lfloor nt \rfloor})\right)_{t\ge 0} \Rightarrow \left( |B_{t}|, |B_{\frac{2}{b_1} t}|\right)_{t\ge 0}
 $$
 where $(B_{t})_{t\ge 0}$ is a standard Brownian motion.
\end{proposition}
\begin{proof}
Theorem A applied to $\Fb^{(r)}$ implies that we have the following convergence in law~: 
\begin{equation*}
\frac{1}{\sqrt{n}}\left(  H^{(r)}( {\lfloor nt \rfloor}),H_f^{(r)}( {\lfloor nt \rfloor})\right)_{t≥ 0} \Rightarrow {2b_1\sqrt{a_1}\over \eta}(a_1b_1)^{-1}\left( |B_{a_1 t}|, |B_t| \right)_{t\ge 0}
\end{equation*}
where  $H_f^{(r)}$ denotes the height function of $\Fb^{(r)}$ restricted to vertices of type $f$. Similarly, Theorem A implies that
\begin{equation*}\label{eq:cvaux2}
\frac{1}{\sqrt{n}}\left(  H^{(w)}( {\lfloor nt \rfloor}),H_f^{(w)}( {\lfloor nt \rfloor})\right)_{t≥ 0} \Rightarrow {2b_1\sqrt{a_1}\over \eta}(a_1b_1)^{-1}\left( |B_{{a_{1}b_{1}\over 2} t}|, |B_t| \right)_{t\ge 0}
\end{equation*}
where  $H_f^{(w)}$ denotes the height function of $\Fb^{(w)}$ restricted to vertices of type $f$. Finally notice that $H_f^{(w)}=H_f^{(r)}$. Using the Brownian scaling with $\frac{a_1b_1}{2}$  yields the theorem with $\sigma^2=\frac{2b_{1}}{\eta^2}=\frac{m(m-1)}{\e[\nu(\nu-1)]}$.
\end{proof}

\section{Proof of Theorem \ref{t:main} }\label{s:proof}
 
 Recall that $\R_{n}$ denotes the set of vertices visited by the walk before time $n$. Let $R_{n}$ be the cardinal of $\R_{n}$. 
 \begin{lemma} \label{l:range}
Let $\eps\in(0,1)$. Let $\mathcal S_{\eps}$ be the event 
$$
\mathcal S_{\eps} := \left\{ \forall\, n\ge 1,\, \sum_{|x|=n} 1 \ge \eps m^n \right\}.
$$
There exists a constant $c>0$ such that for any $a\ge 1$ and $k\ge 1$, 
\begin{equation}\label{eq:range}
\p\left(\mathcal S_{\eps},R_{\tau_{\proot}^{(ka)}} < k^2 \right)\le 3{\rm e}^{-c\eps\sqrt{a}}.
\end{equation}
\end{lemma}

\bigskip

\noindent {\it Proof}. The quenched probability $\P_{\t}\left(R_{\tau_{\proot}^{(ka)}} < k^2 \right)$ is smaller than $\P_{\t}(R_{  \tau_{\proot}^{(1)} }<k^2)^{ka}$. On the event $\left\{ \P_{\t}(R_{\tau_{\proot}^{(1)}}\ge k^2) \ge {1\over k\sqrt{a} }\right\}$, we get that $ \P_{\t}\left(R_{\tau_{\proot}^{(ka)}} < k^2 \right) \le {\rm e}^{-\sqrt{a}}$. Therefore, in order to prove (\ref{eq:range}), we only need to bound the probability 
$$
\GW\left(\mathcal S_{\eps},  \P_{\t}(R_{\tau_{\proot}^{(1)}} \ge k^2) < {1\over k\sqrt{a}} \right).
$$

\noindent For any $x\in\t$, let $(X_{n}^{(x)})_{n\ge 0}$ be the Markov chain starting at $x$. We can couple all $(X_{n}^{(x)},n\ge 0, x\in \t)$ so that $X_{n}^{(x)}$ is the trajectory of $(X_{n})_{n\ge 0}$ after the first visit time at $x$. Let $\mathcal E_{x,k}$ be the event that the walk $(X_{n}^{(x)})_{n\ge 0}$ visits more than $k^2$ distinct vertices before hitting $\px$. Notice that under $\P_{\t}$, the events $(\mathcal E_{x,k},|x|=\ell)$ are mutually independent and independent of $F_{\ell}:=\sigma\{N_{x}^{(1)},|x|\le \ell  \}$. The probability $\P_{\t}(R_{\tau_{\proot}^{(1)}}  \ge k^2)$ is greater than the probability that there exists $x$ such that $N_{x}^{(1)}\ge 1$ and $\mathcal E_{x,k}$ holds. By comparison with a one-dimensional random walk, $\P_{\t}(N_{x}^{(1)}\ge 1) = {m - 1\over m^{|x|+1}-1}$. Using independence, we have for any $\ell\ge 1$,
$$
\P_{\t}(R_{\tau_{\proot}^{(1)}} \ge k^2) \ge {m - 1\over m^{\ell+1}-1} \P_{\t}\left(\bigcup_{|x|=\ell} \mathcal E_{x,k}\right).
$$

\noindent Choosing the largest integer $\ell\ge 1$ such that ${m - 1\over m^{\ell+1}-1} \ge {2\over k\sqrt{a} }$, we get $\P_{\t}(R_{\tau_{\proot}^{(1)}} \ge k^2) \ge  {2\over k\sqrt{a} } \P_{\t}\left(\bigcup_{|x|=\ell} \mathcal E_{x,k}\right)$ hence, for such an $\ell$, 
\begin{eqnarray*}
\GW\left(\mathcal S_{\eps},  \P_{\t}(R_{\tau_{\proot}^{(1)}} \ge k^2) < {1\over k\sqrt{a}} \right) &\le& \GW\left(\mathcal S_{\eps}, \P_{\t}\left(\bigcup_{|x|=\ell} \mathcal E_{x,k}\right)<1/2  \right)\\ 
&\le& 
\GW\left(\sum_{|x|=\ell}  1 \ge \eps m^\ell , \P_{\t}\left(\bigcup_{|x|=\ell} \mathcal E_{x,k}\right)<1/2  \right).
\end{eqnarray*}

\noindent  By independence, we have
$$
\p\left( \bigcap_{|x|=\ell} \mathcal E_{x,k}^c \; \Big|\; \sum_{|x|=\ell}  1 \ge \eps m^\ell \right) \le \p\left(R_{\tau_{\proot}^{(1)}}< k^2 \right)^{\eps m^\ell}.
$$ 
\noindent Notice that the subtree of $\t$ composed of vertices $x$ of type $N_{x}^{(1)}=1$ is a critical Galton--Watson tree with finite variance. It implies that with probability greater than ${c\over k}$ (for some constant $c>0$), the number of vertices of type $1$ is greater than $k^2$. In particular, $\p\left(R_{\tau_{\proot}^{(1)}} < k^2\right) \le 1- {c\over k}$.  It implies that 
$$
\p\left( \bigcap_{|x|=\ell} \mathcal E_{x,k}^c \; \Big|\; \sum_{|x|=\ell}  1 ≥ \eps m^\ell \right) ≤ \ee^{-c'\eps \sqrt{a}}. 
$$
\noindent  By Markov inequality, it yields that 
$$
\GW\left(\P_{\t}\left(\bigcap_{|x|=\ell} \mathcal E_{x,k}^c \right)  > 1/2 \; \Big| \; \sum_{|x|=\ell} 1 \ge \eps m^\ell\right) \le 2{\rm e}^{-c'\eps \sqrt{a}}. 
$$ 

\noindent The proof is complete. $\Box$

\bigskip

The next lemma shows that with high probability we cannot find any long path in the tree with only vertices $y$ of local times greater than $2$.
We call $]\!]  \root,x  [\![$ the set of vertices that lie on the path from the root $\root$ and the vertex $x$, excluding $\root$ and $x$.

\begin{lemma}\label{l:longpath}
There exists a constant $c>0$ such that for any $\ell,k\ge 1$, 
\begin{equation}\label{eq:longpath1}
\p\left(\exists |x|\ge \ell\,:\,  N_{y}^{(k)} \ge 2,\, \forall\, y\in   ]\!]  \root,x  [\![ \; {\rm s.t.} \;\, |y|\in[\ell/2,\ell] \right) \le 2 k^2{\rm e}^{-c \ell} .
\end{equation}
\end{lemma}
\noindent {\it Proof}.  We set by convention $N^{(0)}_{x}=0$ for any $x\in\t$. The event in (\ref{eq:longpath1}) is included in the union of the two following events
\begin{eqnarray*}
 \mathcal E_{1}&:=&\bigcup_{i=1}^k \bigcup_{|x|=\ell} \left\{ \forall\, y\in   ]\!]  \root,x  [\![ \; {\rm s.t.} \;\, |y|\in[\ell/2,\ell], N^{(i)}_{y}-N^{(i-1)}_{y}\ge 2  \right\}, \\ 
 \mathcal E_{2} &:=& \bigcup_{1\le i< j\le k} \bigcup_{|y|=\ell/2} \left\{N^{(i)}_{y}-N^{(i-1)}_{y}\ge 1,  N^{(j)}_{y}-N^{(j-1)}_{y}\ge 1 \right\}.
\end{eqnarray*}

\noindent In words, $\mathcal E_{1}$ is the event that there exists an excursion from the root during which we can find a  path from generation $\ell/2$ to generation $\ell$ on which the walk crossed at least twice every (directed) edge. If this is not the case, it means that  there has been necessarily two excursions from the root which crossed the same vertex at generation $\ell/2$. This is our event $\mathcal E_{2}$. Let us bound both probabilities. By the union bound, we have 
\begin{eqnarray*}
\p(\mathcal E_{1}) &\le& k\p\left(\bigcup_{|x|=\ell}\left\{ \forall\, y\in   ]\!]  \root,x  [\![ \; {\rm s.t.} \;|y|\in[\ell/2,\ell], N_{y}^{(1)}\ge 2 \right \}\right) \\
&\le& k \e\left[ \sum_{|x|=\ell} {\bf 1}_{\{ \forall\, y\in   ]\!]  \root,x  [\![ \;{\rm s.t.}  \; |y|\in[\ell/2,\ell], N_{y}\ge 2 \}}  \right].
\end{eqnarray*}

\noindent By equation (\ref{eq:many-to-one}), the last expectation is 
$$
\e\left[ {1\over \hat N_{\ell}} {\bf 1}_{\{ \forall\, i \in[\ell/2,\ell], \hat N_{i}\ge 2 \}}   \mid \hat N_{0}=1 \right] \le {1\over 2}\p\left( \forall\, i \in[\ell/2,\ell], \hat N_{i}\ge 2 \mid \hat N_{0}=1\right).
$$

\noindent Using a union bound on the last time before time $\ell/2$ when $\hat N_{k}=1$, then simple Markov property, we find that 
$$
\p\left( \forall\, i \in[\ell/2,\ell], \hat N_{i}\ge 2 \mid \hat N_{0}=1 \right) \le (\ell/2) \p\left(\hat \gamma_{1}>\ell/2\right)
$$

\noindent which is exponentially small by Lemma~\ref{l:lyapounov}. This gives the correct upper bound for $\p(\mathcal E_{1})$. Let us bound now the probability of the event $\mathcal E_{2}$. We have (we suppose $\ell$ even for simplicity),
$$
\p(\mathcal E_{2}) \le {k(k-1)\over 2} \p\left(\exists\, |y|=\ell/2\,:\, N_{y}^{(1)}\ge 1,  N^{(2)}_{y}-N_{y}^{(1)}\ge 1\right).
$$

\noindent The probability in the right-hand side is less than 
\begin{equation}\label{eq:majlongpath}
\e\left[ \sum_{|y|=\ell/2}  {\bf 1}_{\{ N_{y}^{(1)}\ge 1,  N^{(2)}_{y}-N_{y}^{(1)}\ge 1  \}}  \right] = \e\left[ \sum_{|y|=\ell/2}  \P_{\t}(N_{y}^{(1)}\ge 1)^2 \right] .
\end{equation}

\noindent We have already seen (by comparison with the biased one-dimensional random walk) that $\P_{\t}(N_{y}\ge 1) = {m-1\over m^{|y|+1} -1}$. Hence, the last expectation is less than $ m^{-\ell/2}$. Therefore $\p(\mathcal E_{2})\le  k^2 m^{-\ell/2}$ and the proof is complete. $\Box$ 

\bigskip 

{\it Proof of Theorem~\ref{t:main}}. Let  $n\ge 1$ and $j_{n}:=\lfloor \sqrt{n}\ln^3(n)\rfloor $. Consider the set of vertices in $\t$ of type $N_{x}^{(j_n)}=1$ and generation greater than ${1\over 2}\ln(n)^2$ and which do not present on their ancestral line any other such vertex (of height greater than ${1\over 2}\ln(n)^2$ and type $1$). In the case in which this set contains more than $n$ vertices, restrict to the $n$ first vertices of the set visited by the walk. Call  $\mathcal L_{n}$ the set of vertices obtained and $L_{n}\le n$  the cardinal of this set. We call $(\tilde \t_{i})_{i\le L_{n}}$ the subtrees rooted at $\mathcal L_{n}$, ordered by the hitting times of their roots by the walk $(X_{k})_{k}$, and for convenience for $i>L_{n}$ we set $\tilde \t_{i}:=  \t_{i}$ (where the $\t_i$ are i.i.d.\ versions of $\t$, and independent of all the random variables introduced so far). Notice that the trees $(\tilde \t_{i})_{i}$ are i.i.d.\ under $\p$.  We call $\mathcal Z_{n}$ the set of vertices in $\t$ which were visited by $(X_k)_{k≤j_n}$ but which do not belong to any of the $(\tilde \t_{i})_{i\le L_{n}}$, see Figure~\ref{f:main}.

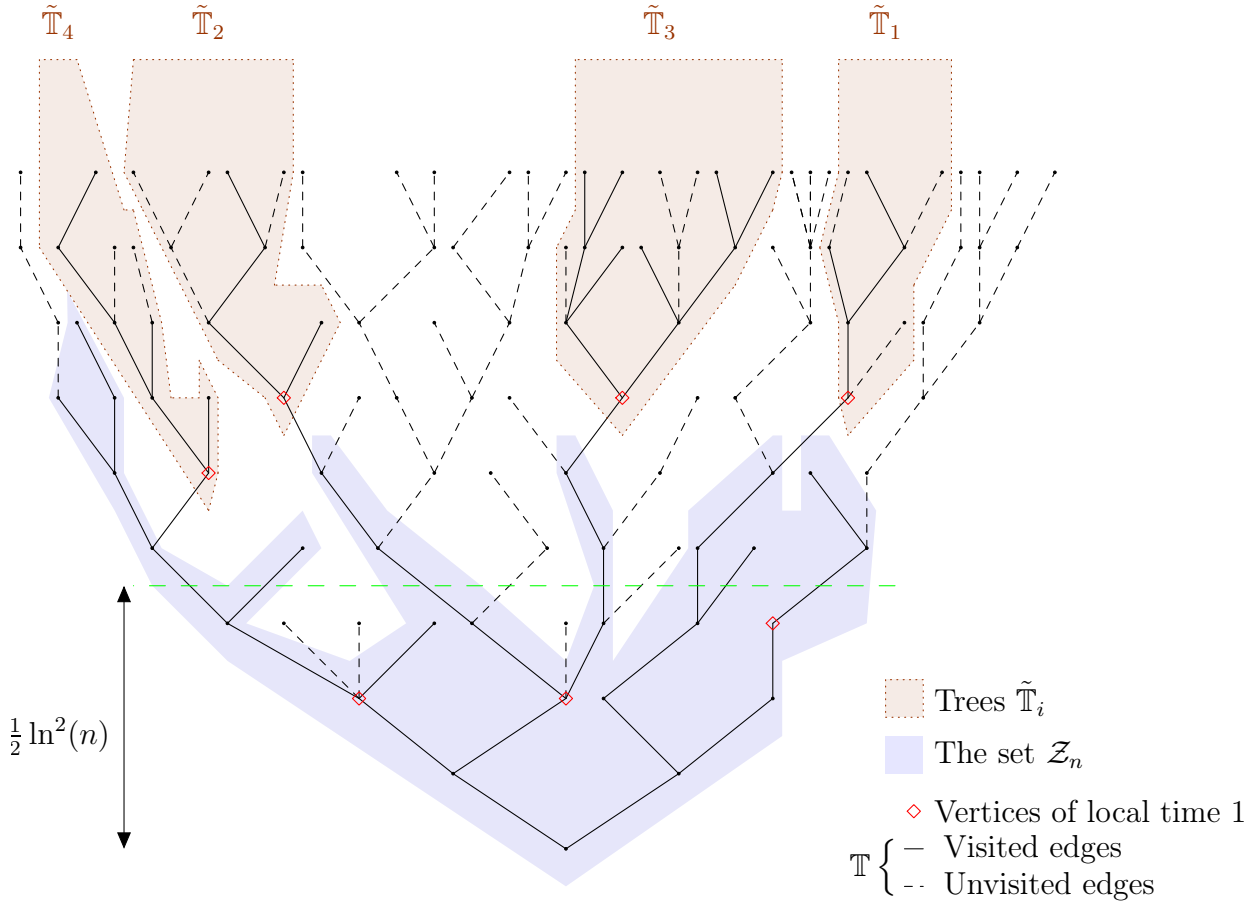
\begin{figure}[H]
\definecolor{zzttqq}{rgb}{0.6,0.2,0}
\definecolor{qqqqcc}{rgb}{0,0,0.8}
\begin{tikzpicture}[line cap=round,line join=round,>=triangle 45,x=0.25cm,y=1.0cm]
\clip(-19.5,-0.75) rectangle (54,11.5);
\fill[line width=0.4pt,dotted,color=zzttqq,fill=zzttqq,fill opacity=0.1] (-5,5.5) -- (-2,7) -- (-3,7.5) -- (-5.5,7.5) -- (-4.5,9) -- (-4.5,10.5) -- (-13,10.5) -- (-13.5,9) -- (-12.5,8.5) -- (-8.5,6.5) -- (-6,6) -- cycle;
\fill[dotted,color=zzttqq,fill=zzttqq,fill opacity=0.1] (-9,4.5) -- (-8.5,5) -- (-8.5,6) -- (-9.5,6.5) -- (-9.5,6) -- (-11,6) -- (-11.5,7) -- (-13,8.5) -- (-13.5,8.5) -- (-16,10.5) -- (-18,10.5) -- (-18,8) -- cycle;
\fill[dotted,color=zzttqq,fill=zzttqq,fill opacity=0.1] (13,5.5) -- (19,7.5) -- (21,8.5) -- (21.5,9) -- (21.5,10.5) -- (10.5,10.5) -- (10.5,8.5) -- (9.5,8) -- (9.5,7) -- (9.5,6.5) -- cycle;
\fill[dotted,color=zzttqq,fill=zzttqq,fill opacity=0.1] (25,5.5) -- (24.5,6) -- (24.5,7) -- (23.5,8) -- (24.5,9) -- (24.5,10.5) -- (30.5,10.5) -- (30.5,8.5) -- (28.5,7.5) -- (28.5,6.5) -- cycle;
\fill[dotted,color=qqqqcc,fill=qqqqcc,fill opacity=0.1] (-17.5,6) -- (-16.5,7) -- (-16.53,7.43) -- (-13.5,6) -- (-13.5,5) -- (-11.5,4) -- (-8,3.5) -- (-4,4.5) -- (-3,4) -- (-7,3) -- (-1.5,2.5) -- (1.5,3) -- (-3.5,5) -- (-3.5,5.5) -- (-2.5,5.5) -- (0.5,4.5) -- (5.5,3.5) -- (10,2.5) -- (11.5,3.5) -- (9.5,5) -- (9.5,5.5) -- (10.5,5.5) -- (12.5,4.5) -- (12.5,2.5) -- (16.5,4) -- (16.5,4.5) -- (21,5.5) -- (21.5,5.5) -- (21.5,4.5) -- (22.5,4.5) -- (22.5,5.5) -- (24,5.5) -- (26.5,4.5) -- (26,3) -- (21.5,2.5) -- (21.5,1.5) -- (10,-0.5) -- (-8,2.5) -- (-12,3.5) -- (-14,4.5) -- cycle;

\draw (10,0)-- (4,1);
\draw (10,0)-- (16,1);
\draw (4,1)-- (10,2);
\draw (4,1)-- (-1,2);
\draw (16,1)-- (12,2);
\draw (16,1)-- (21,2);
\draw (12,2)-- (17,3);
\draw (10,2)-- (5,3);
\draw [dash pattern=on 3pt off 3pt] (10,2)-- (10,3);
\draw (10,2)-- (12,3);
\draw (21,2)-- (21,3);
\draw [dash pattern=on 3pt off 3pt] (-1,2)-- (-5,3);
\draw (-1,2)-- (3,3);
\draw [dash pattern=on 3pt off 3pt] (-1,2)-- (-1,3);
\draw (-1,2)-- (-8,3);
\draw (5,3)-- (0,4);
\draw [dash pattern=on 3pt off 3pt] (5,3)-- (9,4);
\draw [dash pattern=on 3pt off 3pt] (12,3)-- (16,4);
\draw (12,3)-- (12,4);
\draw (17,3)-- (20,4);
\draw (17,3)-- (17,4);
\draw (-8,3)-- (-12,4);
\draw (-8,3)-- (-4,4);
\draw (0,4)-- (-3,5);
\draw [dash pattern=on 3pt off 3pt] (0,4)-- (3,5);
\draw [dash pattern=on 3pt off 3pt] (9,4)-- (6,5);
\draw (12,4)-- (10,5);
\draw [dash pattern=on 3pt off 3pt] (12,4)-- (15,5);
\draw (17,4)-- (21,5);
\draw (21,3)-- (26,4);
\draw (26,4)-- (23,5);
\draw [dash pattern=on 3pt off 3pt] (26,4)-- (26,5);
\draw (-12,4)-- (-9,5);
\draw (-12,4)-- (-14,5);
\draw (-3,5)-- (-5,6);
\draw [dash pattern=on 3pt off 3pt] (-3,5)-- (-1,6);
\draw [dash pattern=on 3pt off 3pt] (3,5)-- (5,6);
\draw [dash pattern=on 3pt off 3pt] (3,5)-- (1,6);
\draw (10,5)-- (13,6);
\draw [dash pattern=on 3pt off 3pt] (10,5)-- (7,6);
\draw [dash pattern=on 3pt off 3pt] (15,5)-- (17,6);
\draw [dash pattern=on 3pt off 3pt] (21,5)-- (19,6);
\draw (21,5)-- (25,6);
\draw [dash pattern=on 3pt off 3pt] (26,5)-- (29,6);
\draw (-14,5)-- (-17,6);
\draw (-14,5)-- (-14,6);
\draw (-9,5)-- (-12,6);
\draw (-9,5)-- (-9,6);
\draw [dash pattern=on 3pt off 3pt] (1,6)-- (-1,7);
\draw [dash pattern=on 3pt off 3pt] (5,6)-- (7,7);
\draw [dash pattern=on 3pt off 3pt] (5,6)-- (3,7);
\draw (13,6)-- (10,7);
\draw (13,6)-- (16,7);
\draw (-5,6)-- (-9,7);
\draw (-5,6)-- (-3,7);
\draw [dash pattern=on 3pt off 3pt] (19,6)-- (23,7);
\draw [dash pattern=on 3pt off 3pt] (25,6)-- (28,7);
\draw (25,6)-- (25,7);
\draw [dash pattern=on 3pt off 3pt] (29,6)-- (32,7);
\draw [dash pattern=on 3pt off 3pt] (29,6)-- (29,7);
\draw (-12,6)-- (-12,7);
\draw (-12,6)-- (-14,7);
\draw (-14,6)-- (-16,7);
\draw [dash pattern=on 3pt off 3pt] (-17,6)-- (-17,7);
\draw [dash pattern=on 3pt off 3pt] (-1,7)-- (-4,8);
\draw [dash pattern=on 3pt off 3pt] (-1,7)-- (3,8);
\draw [dash pattern=on 3pt off 3pt] (7,7)-- (4,8);
\draw [dash pattern=on 3pt off 3pt] (7,7)-- (8,8);
\draw [dash pattern=on 3pt off 3pt] (10,7)-- (10,8);
\draw (10,7)-- (13,8);
\draw (10,7)-- (11,8);
\draw (16,7)-- (14,8);
\draw [dash pattern=on 3pt off 3pt] (16,7)-- (16,8);
\draw (16,7)-- (19,8);
\draw [dash pattern=on 3pt off 3pt] (23,7)-- (21,8);
\draw [dash pattern=on 3pt off 3pt] (23,7)-- (23,8);
\draw (25,7)-- (24,8);
\draw (25,7)-- (28,8);
\draw [dash pattern=on 3pt off 3pt] (29,7)-- (31,8);
\draw [dash pattern=on 3pt off 3pt] (32,7)-- (34,8);
\draw [dash pattern=on 3pt off 3pt] (32,7)-- (32,8);
\draw [dash pattern=on 3pt off 3pt] (-9,7)-- (-11,8);
\draw (-9,7)-- (-6,8);
\draw [dash pattern=on 3pt off 3pt] (-12,7)-- (-13,8);
\draw (-14,7)-- (-17,8);
\draw [dash pattern=on 3pt off 3pt] (-14,7)-- (-14,8);
\draw [dash pattern=on 3pt off 3pt] (-17,7)-- (-19,8);
\draw [dash pattern=on 3pt off 3pt] (-19,8)-- (-19,9);
\draw (-17,8)-- (-15,9);
\draw [dash pattern=on 3pt off 3pt] (-11,8)-- (-13,9);
\draw [dash pattern=on 3pt off 3pt] (-11,8)-- (-9,9);
\draw [dash pattern=on 3pt off 3pt] (-4,8)-- (-4,9);
\draw (-6,8)-- (-8,9);
\draw [dash pattern=on 3pt off 3pt] (-6,8)-- (-5,9);
\draw [dash pattern=on 3pt off 3pt] (3,8)-- (1,9);
\draw [dash pattern=on 3pt off 3pt] (3,8)-- (3,9);
\draw [dash pattern=on 3pt off 3pt] (4,8)-- (7,9);
\draw [dash pattern=on 3pt off 3pt] (8,8)-- (10,9);
\draw [dash pattern=on 3pt off 3pt] (8,8)-- (8,9);
\draw (11,8)-- (11,9);
\draw (11,8)-- (13,9);
\draw [dash pattern=on 3pt off 3pt] (16,8)-- (15,9);
\draw [dash pattern=on 3pt off 3pt] (16,8)-- (17,9);
\draw (19,8)-- (18,9);
\draw (19,8)-- (21,9);
\draw [dash pattern=on 3pt off 3pt] (23,8)-- (22,9);
\draw [dash pattern=on 3pt off 3pt] (24,8)-- (25,9);
\draw [dash pattern=on 3pt off 3pt] (23,8)-- (22,9);
\draw [dash pattern=on 3pt off 3pt] (23,8)-- (23,9);
\draw [dash pattern=on 3pt off 3pt] (23,8)-- (23,9);
\draw [dash pattern=on 3pt off 3pt] (23,8)-- (24,9);
\draw (28,8)-- (26,9);
\draw [dash pattern=on 3pt off 3pt] (28,8)-- (30,9);
\draw [dash pattern=on 3pt off 3pt] (31,8)-- (31,9);
\draw [dash pattern=on 3pt off 3pt] (32,8)-- (34,9);
\draw [dash pattern=on 3pt off 3pt] (32,8)-- (32,9);
\draw [dash pattern=on 3pt off 3pt] (34,8)-- (36,9);
\draw [line width=0.4pt,dotted,color=zzttqq] (-5,5.5)-- (-2,7);
\draw [line width=0.4pt,dotted,color=zzttqq] (-2,7)-- (-3,7.5);
\draw [line width=0.4pt,dotted,color=zzttqq] (-3,7.5)-- (-5.5,7.5);
\draw [line width=0.4pt,dotted,color=zzttqq] (-5.5,7.5)-- (-4.5,9);
\draw [line width=0.4pt,dotted,color=zzttqq] (-4.5,9)-- (-4.5,10.5);
\draw [line width=0.4pt,dotted,color=zzttqq] (-4.5,10.5)-- (-13,10.5);
\draw [line width=0.4pt,dotted,color=zzttqq] (-13,10.5)-- (-13.5,9);
\draw [line width=0.4pt,dotted,color=zzttqq] (-13.5,9)-- (-12.5,8.5);
\draw [line width=0.4pt,dotted,color=zzttqq] (-12.5,8.5)-- (-8.5,6.5);
\draw [line width=0.4pt,dotted,color=zzttqq] (-8.5,6.5)-- (-6,6);
\draw [line width=0.4pt,dotted,color=zzttqq] (-6,6)-- (-5,5.5);
\draw [dotted,color=zzttqq] (-9,4.5)-- (-8.5,5);
\draw [dotted,color=zzttqq] (-8.5,5)-- (-8.5,6);
\draw [dotted,color=zzttqq] (-8.5,6)-- (-9.5,6.5);
\draw [dotted,color=zzttqq] (-9.5,6.5)-- (-9.5,6);
\draw [dotted,color=zzttqq] (-9.5,6)-- (-11,6);
\draw [dotted,color=zzttqq] (-11,6)-- (-11.5,7);
\draw [dotted,color=zzttqq] (-11.5,7)-- (-13,8.5);
\draw [dotted,color=zzttqq] (-13,8.5)-- (-13.5,8.5);
\draw [dotted,color=zzttqq] (-13.5,8.5)-- (-16,10.5);
\draw [dotted,color=zzttqq] (-16,10.5)-- (-18,10.5);
\draw [dotted,color=zzttqq] (-18,10.5)-- (-18,8);
\draw [dotted,color=zzttqq] (-18,8)-- (-9,4.5);
\draw [dotted,color=zzttqq] (13,5.5)-- (19,7.5);
\draw [dotted,color=zzttqq] (19,7.5)-- (21,8.5);
\draw [dotted,color=zzttqq] (21,8.5)-- (21.5,9);
\draw [dotted,color=zzttqq] (21.5,9)-- (21.5,10.5);
\draw [dotted,color=zzttqq] (21.5,10.5)-- (10.5,10.5);
\draw [dotted,color=zzttqq] (10.5,10.5)-- (10.5,8.5);
\draw [dotted,color=zzttqq] (10.5,8.5)-- (9.5,8);
\draw [dotted,color=zzttqq] (9.5,8)-- (9.5,7);
\draw [dotted,color=zzttqq] (9.5,7)-- (9.5,6.5);
\draw [dotted,color=zzttqq] (9.5,6.5)-- (13,5.5);
\draw [dotted,color=zzttqq] (25,5.5)-- (24.5,6);
\draw [dotted,color=zzttqq] (24.5,6)-- (24.5,7);
\draw [dotted,color=zzttqq] (24.5,7)-- (23.5,8);
\draw [dotted,color=zzttqq] (23.5,8)-- (24.5,9);
\draw [dotted,color=zzttqq] (24.5,9)-- (24.5,10.5);
\draw [dotted,color=zzttqq] (24.5,10.5)-- (30.5,10.5);
\draw [dotted,color=zzttqq] (30.5,10.5)-- (30.5,8.5);
\draw [dotted,color=zzttqq] (30.5,8.5)-- (28.5,7.5);
\draw [dotted,color=zzttqq] (28.5,7.5)-- (28.5,6.5);
\draw [dotted,color=zzttqq] (28.5,6.5)-- (25,5.5);

\draw (30,-0.25) node[anchor=east] {$\mathbb{T} \begin{cases} \\ \end{cases}$};

\draw (28,0)-- (29,0);
\draw (29.5,0) node[anchor=west] {Visited edges};

\draw [dash pattern=on 3pt off 3pt] (28,-0.5)-- (29,-0.5);
\draw (29.5,-0.5) node[anchor=west] {Unvisited edges};

\draw [color=red] (28.5,0.5) ++(-2.5pt,0 pt) -- ++(2.5pt,2.5pt)--++(2.5pt,-2.5pt)--++(-2.5pt,-2.5pt)--++(-2.5pt,2.5pt);
\draw (29,0.5) node[anchor=west] {Vertices of local time $1$};

\draw [dotted,color=zzttqq] (27,1.75)-- (29,1.75);
\draw [dotted,color=zzttqq] (29,1.75)-- (29,2.25);
\draw [dotted,color=zzttqq] (29,2.25)-- (27,2.25);
\draw [dotted,color=zzttqq] (27,2.25)-- (27,1.75);
\fill[dotted,color=zzttqq,fill=zzttqq,fill opacity=0.1] (27,1.75) -- (29,1.75) -- (29,2.25) -- (27,2.25) -- cycle;
\draw (29,2) node[anchor=west] {Trees $\tilde{\mathbb{T}}_i$};

\fill[dotted,color=blue,fill=blue,fill opacity=0.1] (27,1) -- (27,1.5) -- (29,1.5) -- (29,1) -- cycle;
\draw (29,1.25) node[anchor=west] {The set $\mathcal{Z}_n$};

\draw [dash pattern=on 6pt off 6pt,color=green] (27.5,3.5)-- (-13,3.5);
\draw [<->] (-13.5,3.5)-- (-13.5,0);
\draw (-17,1.5) node{$\frac{1}{2}\ln^2(n)$};

\draw [color=zzttqq](27,11) node {$\tilde{\mathbb{T}}_1$};
\draw [color=zzttqq](15,11) node {$\tilde{\mathbb{T}}_3$};
\draw [color=zzttqq](-9,11) node {$\tilde{\mathbb{T}}_2$};
\draw [color=zzttqq](-17,11) node {$\tilde{\mathbb{T}}_4$};
\begin{scriptsize}
\draw [fill=black] (10,0) circle (0.5pt);
\draw [fill=black] (4,1) circle (0.5pt);
\draw [fill=black] (16,1) circle (0.5pt);
\draw [color=red] (10,2) ++(-2.5pt,0 pt) -- ++(2.5pt,2.5pt)--++(2.5pt,-2.5pt)--++(-2.5pt,-2.5pt)--++(-2.5pt,2.5pt);
\draw [color=red] (-1,2) ++(-2.5pt,0 pt) -- ++(2.5pt,2.5pt)--++(2.5pt,-2.5pt)--++(-2.5pt,-2.5pt)--++(-2.5pt,2.5pt);
\draw [fill=black] (12,2) circle (0.5pt);
\draw [fill=black] (21,2) circle (0.5pt);
\draw [fill=black] (17,3) circle (0.5pt);
\draw [fill=black] (5,3) circle (0.5pt);
\draw [fill=black] (10,3) circle (0.5pt);
\draw [fill=black] (12,3) circle (0.5pt);
\draw [color=red] (21,3) ++(-2.5pt,0 pt) -- ++(2.5pt,2.5pt)--++(2.5pt,-2.5pt)--++(-2.5pt,-2.5pt)--++(-2.5pt,2.5pt);
\draw [fill=black] (-5,3) circle (0.5pt);
\draw [fill=black] (3,3) circle (0.5pt);
\draw [fill=black] (-1,3) circle (0.5pt);
\draw [fill=black] (-8,3) circle (0.5pt);
\draw [fill=black] (0,4) circle (0.5pt);
\draw [fill=black] (9,4) circle (0.5pt);
\draw [fill=black] (16,4) circle (0.5pt);
\draw [fill=black] (12,4) circle (0.5pt);
\draw [fill=black] (20,4) circle (0.5pt);
\draw [fill=black] (17,4) circle (0.5pt);
\draw [fill=black] (-12,4) circle (0.5pt);
\draw [fill=black] (-4,4) circle (0.5pt);
\draw [fill=black] (-3,5) circle (0.5pt);
\draw [fill=black] (3,5) circle (0.5pt);
\draw [fill=black] (6,5) circle (0.5pt);
\draw [fill=black] (10,5) circle (0.5pt);
\draw [fill=black] (15,5) circle (0.5pt);
\draw [fill=black] (21,5) circle (0.5pt);
\draw [fill=black] (26,4) circle (0.5pt);
\draw [fill=black] (23,5) circle (0.5pt);
\draw [fill=black] (26,5) circle (0.5pt);
\draw [color=red] (-9,5) ++(-2.5pt,0 pt) -- ++(2.5pt,2.5pt)--++(2.5pt,-2.5pt)--++(-2.5pt,-2.5pt)--++(-2.5pt,2.5pt);
\draw [fill=black] (-14,5) circle (0.5pt);
\draw [color=red] (-5,6) ++(-2.5pt,0 pt) -- ++(2.5pt,2.5pt)--++(2.5pt,-2.5pt)--++(-2.5pt,-2.5pt)--++(-2.5pt,2.5pt);
\draw [fill=black] (-1,6) circle (0.5pt);
\draw [fill=black] (5,6) circle (0.5pt);
\draw [fill=black] (1,6) circle (0.5pt);
\draw [color=red] (13,6) ++(-2.5pt,0 pt) -- ++(2.5pt,2.5pt)--++(2.5pt,-2.5pt)--++(-2.5pt,-2.5pt)--++(-2.5pt,2.5pt);
\draw [fill=black] (7,6) circle (0.5pt);
\draw [fill=black] (17,6) circle (0.5pt);
\draw [fill=black] (19,6) circle (0.5pt);
\draw [color=red] (25,6) ++(-2.5pt,0 pt) -- ++(2.5pt,2.5pt)--++(2.5pt,-2.5pt)--++(-2.5pt,-2.5pt)--++(-2.5pt,2.5pt);
\draw [fill=black] (29,6) circle (0.5pt);
\draw [fill=black] (-17,6) circle (0.5pt);
\draw [fill=black] (-14,6) circle (0.5pt);
\draw [fill=black] (-12,6) circle (0.5pt);
\draw [fill=black] (-9,6) circle (0.5pt);
\draw [fill=black] (-1,7) circle (0.5pt);
\draw [fill=black] (7,7) circle (0.5pt);
\draw [fill=black] (3,7) circle (0.5pt);
\draw [fill=black] (10,7) circle (0.5pt);
\draw [fill=black] (16,7) circle (0.5pt);
\draw [fill=black] (-9,7) circle (0.5pt);
\draw [fill=black] (-3,7) circle (0.5pt);
\draw [fill=black] (23,7) circle (0.5pt);
\draw [fill=black] (28,7) circle (0.5pt);
\draw [fill=black] (25,7) circle (0.5pt);
\draw [fill=black] (32,7) circle (0.5pt);
\draw [fill=black] (29,7) circle (0.5pt);
\draw [fill=black] (-12,7) circle (0.5pt);
\draw [fill=black] (-14,7) circle (0.5pt);
\draw [fill=black] (-16,7) circle (0.5pt);
\draw [fill=black] (-17,7) circle (0.5pt);
\draw [fill=black] (-4,8) circle (0.5pt);
\draw [fill=black] (3,8) circle (0.5pt);
\draw [fill=black] (4,8) circle (0.5pt);
\draw [fill=black] (8,8) circle (0.5pt);
\draw [fill=black] (10,8) circle (0.5pt);
\draw [fill=black] (13,8) circle (0.5pt);
\draw [fill=black] (11,8) circle (0.5pt);
\draw [fill=black] (14,8) circle (0.5pt);
\draw [fill=black] (16,8) circle (0.5pt);
\draw [fill=black] (19,8) circle (0.5pt);
\draw [fill=black] (21,8) circle (0.5pt);
\draw [fill=black] (23,8) circle (0.5pt);
\draw [fill=black] (24,8) circle (0.5pt);
\draw [fill=black] (28,8) circle (0.5pt);
\draw [fill=black] (31,8) circle (0.5pt);
\draw [fill=black] (34,8) circle (0.5pt);
\draw [fill=black] (32,8) circle (0.5pt);
\draw [fill=black] (-11,8) circle (0.5pt);
\draw [fill=black] (-6,8) circle (0.5pt);
\draw [fill=black] (-13,8) circle (0.5pt);
\draw [fill=black] (-17,8) circle (0.5pt);
\draw [fill=black] (-14,8) circle (0.5pt);
\draw [fill=black] (-19,8) circle (0.5pt);
\draw [fill=black] (-19,9) circle (0.5pt);
\draw [fill=black] (-15,9) circle (0.5pt);
\draw [fill=black] (-13,9) circle (0.5pt);
\draw [fill=black] (-9,9) circle (0.5pt);
\draw [fill=black] (-4,9) circle (0.5pt);
\draw [fill=black] (-8,9) circle (0.5pt);
\draw [fill=black] (-5,9) circle (0.5pt);
\draw [fill=black] (1,9) circle (0.5pt);
\draw [fill=black] (3,9) circle (0.5pt);
\draw [fill=black] (7,9) circle (0.5pt);
\draw [fill=black] (10,9) circle (0.5pt);
\draw [fill=black] (8,9) circle (0.5pt);
\draw [fill=black] (11,9) circle (0.5pt);
\draw [fill=black] (13,9) circle (0.5pt);
\draw [fill=black] (15,9) circle (0.5pt);
\draw [fill=black] (17,9) circle (0.5pt);
\draw [fill=black] (18,9) circle (0.5pt);
\draw [fill=black] (21,9) circle (0.5pt);
\draw [fill=black] (22,9) circle (0.5pt);
\draw [fill=black] (25,9) circle (0.5pt);
\draw [fill=black] (23,9) circle (0.5pt);
\draw [fill=black] (24,9) circle (0.5pt);
\draw [fill=black] (26,9) circle (0.5pt);
\draw [fill=black] (30,9) circle (0.5pt);
\draw [fill=black] (31,9) circle (0.5pt);
\draw [fill=black] (34,9) circle (0.5pt);
\draw [fill=black] (32,9) circle (0.5pt);
\draw [fill=black] (36,9) circle (0.5pt);
\end{scriptsize}
\end{tikzpicture} 
\caption{ The tree $\t$ (represented up to generation 9), the set $\mathcal{Z}_n$ and the trees $\tilde{\t}_i$. }
\label{f:main}
\end{figure}

By the Kesten-Stigum theorem, we have that  $m^{-n}\sum_{|x|=n} 1$ converges almost surely towards a random variable which is positive on the event of survival of the Galton--Watson tree. By Lemma~\ref{l:range} together with the Borel--Cantelli lemma, it yields that $R_{\tau_{\proot}^{(j_n) }}\ge n$ for $n$ large enough, $\p^*$-almost surely.   It yields that for $\GW^*$-a.e. tree $\t$,  $\lim_{n\to\infty} \P_{\t}(R_{\tau_{\proot}^{(j_n)}}<  n)=0$. Similarly, Lemma~\ref{l:longpath} implies  $\lim_{n\to\infty} \P_{\t}( \max_{x\in\mathcal Z_{n}} |x| > \ln^2(n) )=0$ for $\GW$-a.e. tree $\t$.
Finally, observe that, by comparison with a one-dimensional random walk, for any vertex $x\in \t$ and any integer $k$, $\E_{\t}\left[N_{x}^{(k)}\right] = km^{-|x|}$. By the Markov inequality, we get that $\P_{\t}(\sum_{|x|\le \ln^2(n)} N_{x}^{(j_n)} > j_n\ln^3(n)) \le {1\over \ln^3(n)}\sum_{k\le \ln^2(n)} m^{-k}\sum_{|x|=k}1$ which goes to $0$ $\GW$-a.s.\ as $n\to \infty$. As a result, we can restrict to the event (both for the annealed and quenched convergences in law),
\begin{equation}\label{eq:En}
\mathcal E_{n} :=\left\{R_{\tau_{\proot}^{(j_n)}} \ge  n, \sum_{x\in\mathcal Z_{n}} N_{x}^{(j_n)} \le j_{n}\ln(n)^3,\, \max_{x\in\mathcal Z_{n}} |x|\le \ln^2(n)\right\}.
\end{equation}

\noindent We will repeatedly use the following simple consequence of Lemma 2.3 of~\cite{pitman}, to which we will refer by  (F).  If $(T,d)$ and $(T',d')$ are two real trees and $\varphi:T\to T'$ is a surjective map that sends the root of $T$ to the root of $T'$, then the Gromov-Hausdorff distance $d_{G}(T,T')$ between $T$ and $T'$  is smaller than ${1\over 2} \sup\{|d(x,y) -d'(\varphi(x),\varphi(y))|,(x,y) \in T^2 \}$. \\

On the event $\mathcal E_{n}$, we have $\R_{n} \subset \R_{\tau_{\proot}^{(j_n)}}$. Let $\tilde \f$ be the forest associated to the trees $(\tilde \t_{i})_{i}$. Notice that each tree $\tilde \t_{i}$ is associated with an excursion of the walk  from its root. We call $(\tilde X_{k})_{k\ge 0}$ the concatenation of these excursions, so that $(\tilde X_{ k})_{ k\ge 0}$ restricted to $(\tilde \t_{i})_{i\le L_n}$ mimics the trajectory of the walk $(X_{k})_{k\ge 0}$ in the trees $(\tilde \t_{i})_{i\le L_{n}}$. Finally, for $k\ge 1$, $\tilde \R_{k}$ denotes the set of vertices $\{\tilde X_{\ell},\, \ell\le k\}$. We let $\tilde h(\tilde X_{k})$ be the generation of $\tilde X_{k}$ inside the subtree $\tilde \t_{i}$ to which it belongs. Then, it is sufficient to prove the convergence in law
\begin{equation}\label{eq:cv-tilde}
{1\over \sqrt{\sigma^2 n}} \left( \left\{\tilde h\left(\tilde X_{\lfloor nt\rfloor }\right) \right\}_{t\in [0,1]},\tilde\R_{n}\right) \Rightarrow \left( |{\bf B}|,\T_{|{\bf B}|}  \right)
\end{equation}
both under the annealed probability $\p^*$ and the quenched probability $\P_{\t}$ to prove the theorem. Indeed, let $\tilde n:=\sum_{k\le n} {\bf 1}_{\{X_{k} \notin \mathcal Z_{n}\}}$ be the time spent in the forest $\tilde \f$ until time $n$. On the event $\mathcal E_{n}$, we have $0\le n-\tilde n \le 2j_{n}\ln^3(n)=o(n)$, and $||X_{k}|-\tilde h(\tilde X_{k})| |\le \ln^2(n) + \max_{i\le n-\tilde n} |\tilde h(\tilde X_{k}) - \tilde h(\tilde X_{k-i})|$ which will be $o(n^{1/2})$ in probability uniformly in $k$ if (\ref{eq:cv-tilde}) is proved. On the other hand, $d_{G}(\R_{n},\tilde \R_{\tilde n})\le \ln^2(n)$ (use (F) with $\varphi:\R_{n}\to  \tilde \R_{\tilde n}$ being the identity on $\tilde \f$, and mapping any vertex of $\mathcal Z_{n}$ to the root of $\tilde \R_{\tilde n}$), and $d_{G}(\tilde \R_{\tilde n},\tilde \R_{n})\le \max_{i\le n-\tilde n} |\tilde h(\tilde X_{\tilde n}) -\tilde h(\tilde X_{\tilde n +i})|$ \ (using (F) with $\varphi:\tilde \R_{\tilde n} \to \tilde \R_{n}$ being the identity on $\tilde \R_{\tilde n}$ and mapping the other vertices to $\tilde X_{\tilde n}$) which will be $o(n^{1/2})$ in probability after (\ref{eq:cv-tilde}) is proved. 

\noindent {\it The annealed case}

 It is actually enough to prove (\ref{eq:cv-tilde}) under $\p^*_{n_{0}}:=\p(\cdot \mid \sum_{|x|=n_{0} } 1 >0)$ for any integer $n_{0}$.  Notice that the trees $(\tilde \t_{i})_{i}$ are still i.i.d.\ under $\p^*_{n_{0}}$ as soon as $n$ is such that ${1\over 2} \ln(n)^2 \ge n_{0}$ (this was not true under $\p^*$). To a tree $\tilde \t_{i}$, we associate the trees $(\tilde{\Tb}_{i}, \tilde \Tb_{i}^{(w)}, \tilde \Tb_{i}^{(r)})$ as in Section~\ref{s:reduction}. The concatenation of these trees yields the forests $(\tilde{\Fb}, \tilde \Fb^{(w)},\tilde \Fb^{(r)})$. Let $\tilde H^{(w)}, \tilde H^{(r)}$ be the height functions associated to the forests $\tilde \Fb^{(w)}$ and $\tilde \Fb^{(r)}$. Recall that by construction we have $\tilde H^{(w)}(k)=\tilde h(\tilde X_{k})$ for any $k\ge 0$. 
By an abuse of notation, we denote by $\tilde \Fb^{(r)}\cap \tilde \R_{n}$ the forest $\tilde \Fb^{(r)}$ restricted to vertices which were in $\tilde \R_{n}$ before the reduction of Section~\ref{s:reduction}. We see that $d_{G}(\tilde \R_{n},  \tilde \Fb^{(r)}\cap \tilde \R_{n}) \le  \tilde \ell(n)$, where $\tilde \ell(n)$ is the maximal length of an edge explored by the depth-first search in the forest $\tilde \Fb^{(r)}$ before visiting $n$ vertices of type $1$.  By Lemma~\ref{l:loic} (ii), we see that $d_{G}(\tilde \R_{n},  \tilde \Fb^{(r)}\cap \tilde \R_{n})$ is $o(\sqrt{n})$ in probability. We want to show the convergence of $\tilde \Fb^{(r)}\cap \tilde \R_{n}$ for the Gomov-Hausdorff topology. This is equivalent with showing the convergence of its contour function, see Lemma 2.3 of~\cite{DuLG05}. In our case, the convergence of the contour function would be implied by that of the height function of the tree. Indeed, the argument used in the proof of Theorem~2.4.1 of~\cite{DuLG02} can be adjusted to height/contour functions of trees with edge lengths, as inequality~(2.34) of~\cite{DuLG02} can be adjusted to such functions and inequality~(2.35) remains valid. As a result, we need to show that under $\p^*_{n_{0}}$, (with $\tilde R_{n}$ being the cardinal of $\tilde \R_{n}$),
$$
{1\over \sqrt{\sigma^2 n}} \left( \left\{ \tilde H^{(w)}\left({\lfloor  nt \rfloor}\right)\right\}_{t\in [0,1]},\left\{ \tilde H^{(r)}\left(\lfloor \tilde R_{n} t\rfloor \right)\right\}_{t\in [0,1]} \right) \Rightarrow \left( |{\bf B}| ,  {|{\bf B}|} \right).
$$

\noindent  This  follows from Proposition~\ref{p:loic} once we prove that ${\tilde R_{n} \over n}$ converges towards $\frac{b_1}{2}$ in probability.  Let $\tilde u^{(w)}(i)$, $\tilde u^{(r)}(i)$ be the indices in the forests $\tilde \Fb^{(w)}$ and $\tilde \Fb^{(r)}$ of the $i$-th vertex of type $1$ for the first-depth search order. Let $i_{n}$ be the number of vertices of type $1$ that was visited by $(\tilde X_{k})_{k}$ before time $n$. Since $\tilde H^{(w)}$ mimics the walk $\tilde X$, we deduce that $\tilde u^{(w)}(i_{n})$ and $\tilde u^{(w)}(i_{n}+1)$ are lower and upper bounds of $n$. By Lemma~\ref{l:loic} (i), it implies that $i_{n}/n$ converges to ${a_{1}b_{1}\over 2}$ in probability.  On the other hand, $\tilde R_{n}$ is between $\tilde u^{(r)}(i_{n})$ and $\tilde u^{(r)}(i_{n}+1)$. Lemma~\ref{l:loic} (i) yields that ${\tilde R_{n} \over i_{n}}$ converges to ${1\over a_{1}}$ in probability. The proof is complete in the annealed case. Observe that the same proof works under $\p$ which means that (\ref{eq:cv-tilde}) also holds under $\p$.\\

\noindent {\it The quenched case}

Let $\Xi_{n}:= {1\over \sqrt{\sigma^2 n}} ( \tilde \R_{n}, \{ \tilde h(\tilde X_{\lfloor nt\rfloor })\}_{t\in [0,1]})$. Let $F$ be some bounded nonnegative continuous function on the space of real trees times c\`adl\`ag functions on $[0,1]$. We know by (\ref{eq:cv-tilde}) (used in the annealed setting under $\p$) that $\e\left[F(\Xi_{n})\right]$ converges. We want to show  that $\E_{\t}\left[ F(\Xi_{n}) \right]$ converges to the same limit 
for $\GW$-a.e. every tree $\t$. Hence we are going to show that $\E_{\t}\left[ F(\Xi_{n})  \right]$ concentrates around its mean value. To do this, we will show that the variance decreases fast enough. Let $\mathcal V$ be the space of all  finite sets of words in $\bigcup_{k\ge 0} {\mathbb N}^k$. Let $\mathcal A$ be the set of couples $(V_{1},V_{2})\in \mathcal V^2$  such that no vertex of $V_{1}$ is an ancestor of a vertex in $V_{2}$ and vice-versa (it includes all couples which contain $\emptyset$). Recall that by construction if $L_{n}=0$, then $\E_{\t}\left[ F(\Xi_{n}) \right] = \e\left[F(\Xi_{n})\right]$. We compute that 
$$
E_{\GW}\left[  \E_{\t}\left[ F(\Xi_{n})\right] ^2 \right] = \sum_{(V_{1},V_{2}) \in \mathcal V^2 } E_{\GW}\left[  \E_{\t}\left[ F(\Xi_{n}){\bf 1}_{\{ \mathcal L_{n}= V_{1} \}} \right]   \E_{\t}\left[ F(\Xi_{n}){\bf 1}_{\{ \mathcal L_{n}= V_{2} \}} \right]  \right].
$$

\noindent We divide the sum in two, depending on whether $(V_{1},V_{2})$ belongs to $\mathcal A$ or not. We observe that, 
\begin{eqnarray*}
 && \sum_{(V_{1},V_{2})\in \mathcal A } E_{\GW}\left[  \E_{\t}\left[ F(\Xi_{n}){\bf 1}_{\{ \mathcal L_{n}= V_{1} \}} \right]   \E_{\t}\left[ F(\Xi_{n}){\bf 1}_{\{ \mathcal L_{n}= V_{2} \}} \right]  \right]\\
&=&
\sum_{(V_{1},V_{2})\in \mathcal A  } E_{\GW}\left[  \P_{\t}(\mathcal L_{n}=V_{1})\P_{\t}(\mathcal L_{n}=V_{2}) \right]   \e\left[ F(\Xi_{n}) \right]^2  \\
&\le& 
\e\left[ F(\Xi_{n})\right]^2 
\end{eqnarray*}

\noindent by the branching property. For the rest of the sum, we just write since $F$ is bounded by, say, some $M$,
\begin{eqnarray*}
&& \sum_{(V_{1},V_{2}) \notin \mathcal A  } E_{\GW}\left[  \E_{\t}\left[ F(\Xi_{n}){\bf 1}_{\{ \mathcal L_{n}= V_{1} \}} \right]   \E_{\t}\left[ F(\Xi_{n}){\bf 1}_{\{ \mathcal L_{n}= V_{2} \}} \right]  \right] \\
&\le&
M^2 \sum_{(V_{1},V_{2}) \notin \mathcal A  } E_{\GW}\left[\P_{\t}(\mathcal L_{n}= V_{1})\P_{\t}(\mathcal L_{n}= V_{2})  \right].
\end{eqnarray*}

\noindent We end up with 
$$
E_{\GW}\left[ \E_{\t}\left[ F(\Xi_{n}) \right] ^2 \right] - \e\left[ F(\Xi_{n})\right]^2 \le M^2 \sum_{(V_{1},V_{2}) \notin \mathcal A  } E_{\GW}\left[\P_{\t}\left(\mathcal L_{n}= V_{1}\right)\P_{\t}\left(\mathcal L_{n}= V_{2}\right)  \right].
$$

\noindent It is enough to show that the right-hand side is summable in $n$. We have 
$$
\sum_{(V_{1},V_{2}) \notin \mathcal A  } E_{\GW}\left[\P_{\t}(\mathcal L_{n}= V_{1})\P_{\t}\left(\mathcal L_{n}= V_{2}\right)  \right] = \sum_{V_{1} \in \mathcal V  } E_{\GW}\left[\P_{\t}\left(\mathcal L_{n}= V_{1} \right)\P_{\t}\left( (\mathcal L_{n}, V_{1} )\notin \mathcal A\right)  \right].
$$ 

\noindent Notice that we can restrict to $V_{1}$ such that $\min_{x\in V_{1}} |x| \ge {1\over 2}\ln^2(n)$ and with cardinal smaller than $n$. Let us bound the probability $\P_{\t}( (\mathcal L_{n}, V_{1} )\notin \mathcal A)$, for such a set $V_{1}$. If $(\mathcal L_{n}, V_{1} )\notin \mathcal A$, it means that the random walk $(X_{k})_{k}$  has visited before time $\tau_{\proot}^{(j_{n})}$ the ancestor at generation ${1\over 2} \ln^2(n)$ of a vertex in $V_{1}$. This probability is smaller than $j_{n} m^{-{1\over 2} \ln^2(n)}$, and there is at most $n$ such ancestors. Therefore, $\P_{\t}( (\mathcal L_{n}, V_{1} )\notin \mathcal A)\le  nj_{n}m^{-{1\over 2} \ln^2(n)} $ and we deduce that 
\begin{equation}\label{eq:quenched}
\sum_{(V_{1},V_{2}) \notin \mathcal A  } E_{\GW}\left[\P_{\t}\left(\mathcal L_{n}= V_{1}\right)\P_{\t}\left(\mathcal L_{n}= V_{2}\right)  \right] \le  nj_{n} m^{-{1\over 2} \ln^2(n)}
\end{equation}

\noindent which is summable indeed. $\Box$

\section{Random walks in random environment on a Galton--Watson tree}
\label{s:rwre}
Let $(V(x))_{x\in\t}$ be a branching random walk on $\mathbb R$ and $\mathbb{V}:=(\t, (V(x))_{x\in\t})$. More specifically, we let $V(\proot)=V(\root)=0$. At generation $n$, and conditionally on $\sigma\{V(y),\, |y|\le n\}$, the random variables $(V(xi)-V(x),i\le \nu(x))_{|x|=n}$ are supposed to be independent and identically distributed. The common law does not depend on $n$. Conditionally on $\mathbb{V}$, we consider the Markov chain $(X_{n})_{n\ge 0}$ on $\t$ such that for any $x\neq \proot$,
\begin{eqnarray}
\P_{\v}(X_{n+1}=\px\,|\, X_n=x) &=& {{\rm e}^{-V(x)} \over {\rm e}^{-V(x)} + \sum_{i=1}^{\nu(x)} {\rm e}^{-V(xi)} }, \label{p(x,px)2}\\
\P_{\v}(X_{n+1}=xi\,|\, X_n=x) &=& {{\rm e}^{-V(xi)} \over {\rm e}^{-V(x)} + \sum_{i=1}^{\nu(x)} {\rm e}^{-V(xi)} } \; \; \mbox{for any} \;\; 1\le i\le \nu(x) \label{p(x,xi)2},
\end{eqnarray} 

\noindent and which is reflected on $\proot$. The biased random walk is the particular case $V(x) = |x| \ln(\lambda)$. We let $\p$ be the measure $\P_{\v}$ integrated over the law of $\v$. They are associated to the expectations $\e$ and $\E_{\v}$. The measure of $\v$ will be denoted by ${\rm BW}$, and by ${\rm BW}^*$ when conditioned upon the event that $\t$ is infinite. They are associated to the expectations $E_{\BW}$ and $E_{\BW^*}$. We introduce 
$$
\psi(t):= E_{\BW}[\sum_{|x|=1} {\rm e}^{-tV(x)}].
$$

\noindent   Lyons and Pemantle~\cite{LP92} showed that the Markov chain is recurrent or transient depending on whether $\min_{t\in[0,1]} \psi(t)$ is  respectively $\le 1$ or $> 1$ (moreover it is positive recurrent in the case $<1$). We consider the critical  case $\min_{t\in[0,1]} \psi(t)=1$. In the papers~\cite{HuShi07}, \cite{HuShi07'}, \cite{HuShi15}, Hu and Shi proved that when $\psi'(1)\ge 0$ the random walk is of order $\log(n)^3$   whereas when $\psi'(1)<0$, it is of order $n^\nu$ where $\nu:= 1-{1\over \min(\kappa,2)}$ and $\kappa:=\inf\{ t>1\,:\, \psi(t)=1\}\in (1,+\infty]$. In the latter case, Andreoletti and Debs~\cite{andreoletti-debs} showed that the local time of the root at time $n$ was of order $n^{1/\min(\kappa,2)}$, and that the largest generation entirely visited at time $n$ was of order $\ln(n)$. Then, in~\cite{Hu}, Hu showed that the local time was actually converging in law after a suitable rescaling. \\
In the case $\kappa>2$, the walk is therefore of order $\sqrt{n}$ and we expect  a central limit theorem. Indeed Faraud~\cite{faraud} generalized the result of Peres and Zeitouni~\cite{PeZe06} and showed a central limit theorem, at least when $\kappa>5$ in the annealed case, and $\kappa>8$ in the quenched case. We extend this result to the convergence of the trace of the random walk to the Brownian forest, with the condition $\kappa>2$. At the present time, no central limit theorem has been shown for $\kappa≤2$. We keep the notation $\R_{n}$, $\T_{|{\bf B}| }$ of the introduction. \\

\begin{theorem}\label{t:rwre}
Assume that $\psi(1)=1$, $\psi(2)<1$ and $\e\Big[ \sum_{|x|=|y|=1,x\neq y} \ee^{-V(x)}\ee^{-V(y)} \Big]<\infty$.  Let 
$$
\sigma^2:= (1-\psi(2)) / \e\Big[\sum_{\substack{|x|=|y|=1 \\ x\neq y}} \ee^{-V(x)}\ee^{-V(y)}\Big].
$$
 Under $\p^*$ (annealed case) and under $\P_{\v}$ for ${\rm BW}^*$-a.e. $\v$(quenched case), the following joint  convergence in law holds as $n\to\infty$:
$$
{ 1\over \sqrt{  \sigma^2 n}} \left(\left\{   |X_{\lfloor nt \rfloor }| \right\}_{t\in [0,1]},\R_{n}\right) \Rightarrow \left(|{\bf B} |,\T_{|{\bf B}| } \right)
$$
for the  Skorokhod topology on the space of c\`adl\`ag functions and the Gromov-Hausdorff topology on the space of real trees.
\end{theorem}

The proof follows the same lines as for Theorem~\ref{t:main}. We proceed by adapting the steps of the proof to our case.

\subsection{Description of the process of local times}

We adapt Section~\ref{s:multi-type}. Lemma~\ref{l:multi-type} still holds. The mean matrix is now given by
$$
m_{i,j}:=\e\left[\sum_{|x|=1} {\bf 1}_{\{N_{x}=j\}} \mid N_{\root}=i\right]  =  \binom{i+j-1}{j} \e\left[ \sum_{|x|=1} {{\rm e }^{-jV(x)} \over (1+{\rm e}^{-V(x)})^{i+j}}\right].
$$

We introduce a random variable $\hat S_{1}$ with law given by $\e[f(\hat S_{1})] = \e\left[\sum_{|x|=1} f(V(x)){\rm e}^{-V(x)}  \right]$. Notice that under the assumptions of Theorem~\ref{t:rwre} we have  $\e\left[\hat S_{1}\right]>0$ (it may be infinite if the mean number of children in $\t$ is infinite). We define $(\hat S_{k})_{k\ge 1}$ as the random walk with step distribution given by the law of $\hat S_{1}$. Let
\begin{eqnarray*}
a_{i} &:=& \e\left[  {\left(\sum_{\ell \ge 1} {\rm e}^{-\hat S_{\ell}}\right)^{i-1}   \over \left(1+\sum_{\ell \ge 1} {\rm e}^{-\hat S_{\ell}}\right) ^{i+1}} \right] / \e\left[ {1\over 1 + \sum_{\ell \ge 1} {\rm e}^{-\hat S_{\ell}}}\right],\\
b_{i} &:=& i\e\left[ {1\over 1 + \sum_{\ell \ge 1} {\rm e}^{-\hat S_{\ell}}}\right].
\end{eqnarray*}

\begin{lemma}\label{l:hm}
The vectors $(a_i)_{i\in\N^*}$ and $(b_i)_{i\in\N^*}$ are the left and right eigenvectors associated to the eigenvalue $1$ of the mean matrix $(m_{i,j})_{i,j\geq 1}$. Moreover, they are normalized so that \mbox{$\sum_{i} a_{i}=1$} and $\sum_{i} \pi_{i}=1$ (where $\pi_{i}:=a_{i}b_{i}$).
\end{lemma}

\begin{proof}
Simple computations ensure that both normalizations are right, and that $(b_i)_{i\in\N^*}$ is a right eigenvector of the mean matrix. Let $C_a:=\e\Big[(1+\sum_{\ell≥1} \ee^{-\hat{S}_\ell})^{-1}\Big]$ and suppose that the random walk $(\hat S_{k})_{k\ge 1}$ is taken independent of $\v$. We have for all $j\in\N^*$,
\begin{align*}
C_a\sum_{i≥1} a_i m_{i,j}
&=
\e\left[\frac{(\sum_{\ell ≥ 1} \ee^{-\hat S_{\ell}})^{-1}}{1+\sum_{\ell ≥ 1} \ee^{-\hat S_{\ell}}} \sum_{|x|=1}\sum_{i≥1} \binom{i+j-1}{j} \Big(\frac{\sum_{\ell ≥ 1} \ee^{-\hat S_{\ell}}}{1+\sum_{\ell ≥ 1} \ee^{-\hat S_{\ell}}}\Big)^i {{\rm e }^{-jV(x)} \over (1+\ee^{-V(x)})^{i+j}}\right] \\
&=\e\left[\sum_{|x|=1} \frac{\Big(\ee^{-V(x)}(1+\sum_{\ell≥1}\ee^{-\hat{S}_{\ell}})\Big)^{j-1}}{\Big(1+\ee^{-V(x)}(1+\sum_{\ell≥1}\ee^{-\hat{S}_{\ell}})\Big)^{j+1}}\ee^{-V(x)}\right] \\
&=\e\left[\frac{\Big(\ee^{-\hat{S}_1'}(1+\sum_{\ell≥1}\ee^{-\hat{S}_{\ell}})\Big)^{j-1}}{\Big(1+\ee^{-\hat{S}_1'}(1+\sum_{\ell≥1}\ee^{-\hat{S}_{\ell}})\Big)^{j+1}}\right]
\end{align*} 
where $\hat{S}_1'$ has the same law than $\hat{S}_1$, and is independent of $(\hat{S}_\ell)_{\ell≥1}$. This yields 
\begin{equation*}
C_a\sum_{i≥1} a_i m_{i,j}=\e\left[\frac{\Big(\sum_{\ell≥1}\ee^{-\hat{S}_{\ell}'}\Big)^{j-1}}{\Big(1+\sum_{\ell≥1}\ee^{-\hat{S}_{\ell}'}\Big)^{j+1}}\right]=C_a a_j,
\end{equation*} 
where we set for all $\ell>1$ $\hat{S}_\ell':= \hat S_{1}' +  \hat{S}_{\ell-1}$, and used the fact that with this setting $(\hat{S}_\ell)_{\ell≥1}$ has the same law than $(\hat{S}_\ell')_{\ell≥1}$. 
\end{proof}

Now the many-to-one lemma states that for any bounded function $f:\mathbb{N}^n\to\mathbb{R}$, we have 
$$
\e\left[\sum_{|x|=n,N_{x}\ge 1} f(N_{x_{1}},N_{x_{2}},\ldots,N_{x_{n-1}},N_{x})   \right] 
=
\e\left[{1\over \hat N_{n}}f(\hat N_{1},\hat N_2,\ldots,\hat N_{n-1},\hat N_{n})\right]
$$

\noindent where $(\hat N_{i})_{i\ge 0}$ is a Markov chain on $\mathbb{N}^*$ starting at $1$ and with transition probabilities from $i$ to $j$ given by $m_{i,j}{b_{j}\over b_{i}}$. We adapt Lemma~\ref{l:lyapounov}.   
\begin{lemma}\label{l:lyapounov2}
Let $\hat \gamma_{1}:=\min\{i\ge 1\,:\, \hat N_{i} =1\}$. There exists $r>0$ such that $\e\left[{\rm e}^{r\hat \gamma_{1}}\right]<\infty$.
\end{lemma}
\begin{proof} We set for all $i\geq 1$, $F(i):=i$. A computation leads to 
\begin{equation*}
\sum_{j≥1}\hat{p}_{i,j}F(j)=\e\left[ \sum_{|x|=1}\ee^{-V(x)} \right]+\e\left[ \sum_{|x|=1}\ee^{-2V(x)} \right](i+1).
\end{equation*}
Now since $\psi(2)<1$,  there exists $d<1$ such that for all $i>i_0$ large enough, $\sum_{j≥1}\hat{p}_{i,j}F(j)<dF(i)$. We conclude as in Lemma~\ref{l:lyapounov}.
\end{proof}

\subsection{Reduction of trees }

We adapt Section~\ref{s:reduction}. We define in the same way the trees $\Tb$, $\Tb^{(r)}$ and $\Tb^{(w)}$ there. Lemma~\ref{l:lyapounov2} ensures that $\Tb^{(r)}$ satisfies conditions (ii) and (iii) of Section~\ref{s:preliminaries}. The constants $\Sigma_{f}$, $C_{1}$ and $C_{2}$ in Section~\ref{s:preliminaries} associated to $\Tb^{(r)}$ are given by  $\Sigma_{f}={\eta \over b_{1}\sqrt{a_{1}}}$, $C_{1}=a_{1}$, $C_{2} ={1\over a_{1}b_{1}}$, and $\eta>0$ is given by 
$$
\eta^2={2\over 1-\psi(2)} \e\Big[\Big( 1+\sum_{\ell≥1} \ee^{-\hat{S}_{\ell}} \Big)^{-1}\Big]\e\Big[\sum_{\substack{|x|,|y|=1 \\ x\neq y}} \ee^{-V(x)}\ee^{-V(y)}\Big]
$$
 (in the setting and notation of~\cite{loic}, we have for all $i,j,k\in\N^*$, 
 $$
 Q_{i,j}^{k}=\binom{i+j+k-1}{i,j,k-1}\e\Big[ \sum_{|x|=1,|y|=1,x\neq y} \frac{\ee^{-iV(x)}\ee^{-jV(y)}}{(1+\ee^{-V(x)}+\ee^{-V(y)})^{i+j+k}} \Big]).
 $$

\bigskip

Similarly, the tree $\Tb^{(w)}$ is associated to the constants $\Sigma_{f}={\eta \over b_{1}\sqrt{a_{1}}}$, $C_{1}={2\over a_{1}b_{1}}$ and $C_{2} ={1\over a_{1}b_{1}}$. Proposition~\ref{p:loic} still holds (with our choice of $b_{1}$ and $\sigma$). The proof remains unchanged.

\section{Proof of Theorem \ref{t:rwre}}

 We adapt Section~\ref{s:proof}. We give here an analogue of Lemma~\ref{l:range}. This analogue is less precise but is still sufficient for our purpose.  
 \begin{lemma} \label{l:range2}
For $k\ge 1$ large enough, $R_{\tau_{\proot}^{(k\fl{\ln^{10}(k)})}} \ge k^2$
$\BW^*$-a.s.
\end{lemma}

\bigskip

\noindent {\it Proof}. Define the set of vertices $G_{k}$ which contains all vertices $x\in \t$ such that $ \ee^{V(x)}>k\ln^2(k)$ and $\ee^{V(y)}<k\ln^2(k)$ for any strict ancestor $y$ of $x$.  In other words, $G_k$ is the set of vertices which are the first of their ancestry line to be such that $\ee^{V(x)}>k\ln^2(k)$. First let us collect some few facts about $G_{k}$. Notice that the $G_k$ are simple optional lines increasing in $k$, as defined in~\cite{BiKy}. Applying Theorem~6.1 of~\cite{BiKy} to $\sum_{x\in G_k} \ee^{-V(x)}$ and using the fact that $\sum_{|x|=n} \ee^{-V(x)}$ converges $\BW^*$-a.s.\ to a positive random variable (see~\cite{lyons97}), we get that $\BW^*$-a.s, there exists $\eps>0$ small enough so that the event
\begin{equation*}
\mathcal{S}_\eps=\left\{ \forall k\geq 1,\: \sum_{x \in G_{k}} \ee^{-V(x)}>\eps \right\}
\end{equation*}

\noindent holds.  Furthermore, for any real number $b>0$, we have
\begin{align*}
\p\Big( \exists x\in G_k\: :\: |x|>\fl{\ln^2(k)} \Big)&\leq \e\Big[ \sum_{|x|=\fl{\ln^2(k)}} \1{\ee^{V(x)}<k\fl{\ln^2(k)}} \Big]\\
&\leq (k\fl{\ln^2(k)})^{b} \e\Big[ \sum_{|x|=\fl{\ln^2(k)}} \ee^{-b V(x)} \Big]=(k\fl{\ln^2(k)})^{b}\psi(b)^{\fl{\ln^2(k)}}. 
\end{align*}
Taking $b$ such that $\psi(b)<1$, this last quantity is summable in $k$ and the Borel-Cantelli lemma gives us that {\rm BW}-a.s., for $k$ large enough, for all $x\in G_k$, $|x|<\fl{\ln^2(k)}$. 
Finally, considering $\t^+$ the set of vertices $x\in\t$ such that $V(y)<V(x)$ for any strict ancestor $y$ of $x$,  notice that the tree induced by $(\t^+,(V(x))_{x\in\t^+})$ is a branching random walk with positive increments along paths in $\t^+$. More precisely, setting for any $x\in\t^+$, $\sigma_x:=V(x)$ and $\lambda_x=\infty$, it is a C-M-J process as defined in~\cite{Nerman} with Malthusian parameter $\alpha=1$. 
Applying Theorem~6.3 of~\cite{Nerman} to characteristics $\phi_x(t):=\1{t>0}\sum_{\parent{y}=x}\1{\sigma_{y}-\sigma_x >t}$ and $\psi_x(t):=\1{t>0}\sum_{\parent{y}=x}\ee^{t-(\sigma_y-\sigma_x)}\1{\sigma_{y}-\sigma_x>t}$ (conditions 6.1 and 6.2 there being satisfied with $\beta=0$), we get that $\BW^*$-a.s.\
\begin{equation*}
\frac{ k\ln^2(k)\sum_{x\in G_k} \ee^{-V(x)}}{\sum_{x\in G_k} 1}=\frac{\sum_{x\in\t^+} \psi_x( \ln(k\fl{\ln^2(k)})-\sigma_x)}{\sum_{x\in\t^+} \phi_x(\ln(k\fl{\ln^2(k)})-\sigma_x)} \sto{k\to\infty} C\in(0;1],
\end{equation*}

\noindent where $C$ is a positive deterministic constant. Let $G_k':=\{x\in G_k\: :\: e^{V(x)}< {2\over C} k \ln^2(k) \}$. We observe that $k\ln^2(k)\sum_{x\in G_k} \ee^{-V(x)} \le \#G'_{k} + {C\over 2} (\#G_{k}-\#G'_{k})$ which implies that $\# G'_{k} \ge c \eps k \ln^2(k)$ for $k$ large enough $\BW^*$-a.s.\ on the event $\mathcal S_{\eps}$, where $c$ is any positive constant smaller than ${1\over 2-C}$.  

Let us now proceed to the proof. We follow the lines of the proof of Lemma~\ref{l:range} with the setting $a=\fl{\ln^{10}(k)}$. We only need to show that $\BW$-as on the event $\mathcal S_{\eps}$, we have 
$$
\P_{\v}(R_{\tau_{\proot}^{(1)}} \ge k^2) \ge {1\over k\fl{\ln^5(k)}}
$$

\noindent for $k$ large enough. We define again $(X_{n}^{(x)})_{n\ge 0}$ as the Markov chain starting at $x$ and $\mathcal E_{x,k}$  the event that the walk $(X_{n}^{(x)})_{n\ge 0}$ visits more than $k^2$ distinct vertices before hitting $\px$. By comparison with a one-dimensional random walk, $\P_{\v}(N_{x}^{(1)}\ge 1)$ is now equal to $\left( 1 +{\rm e}^{V(x_{1})} + \ldots + {\rm e}^{V(x)}\right)^{-1}$ which is greater than $\frac{1}{(2/C)k\fl{\ln^2(k)}\times (|x|+1)}\geq\frac{2}{\fl{\ln(k)}k\fl{\ln^2(k)}\times \fl{\ln^2(k)} }$ if $x\in G_k'$ for $k$ large enough. We deduce that {\rm BW}-a.s.\ on the event $\mathcal S_{\eps}$, for $k\geq 1$ large enough,
$$
\P_{\v}(R_{\tau_{\proot}^{(1)}} \ge k^2) \ge \frac{2}{k\fl{\ln^5(k)}}\P_{\v}\left(\bigcup_{ x \in G_k'} \mathcal E_{x,k}\right).
$$

\noindent We see that we simply need that $\P_{\v}\left(\bigcup_{ x \in G_k'} \mathcal E_{x,k}\right) \ge {1\over 2}$ for $k$ large enough, $\BW$-as on the event $\mathcal S_{\eps}$. We showed that $\#G'_{k} \ge c\eps k \ln^2(k)$  for $k$ large enough $\BW$-a.s.\ on the event $\mathcal S_{\eps}$. Hence, the proof will be complete if we prove that the probabilities
$$
\BW\left( \# G_k'> c \eps k\ln^2(k)\:\textrm{and}\: \P_{\v}\left(\bigcup_{x \in G'_{k}} \mathcal E_{x,k}\right)<1/2  \right).
$$
are summable in $k$. This is done by following the final lines  of the proof of Lemma~\ref{l:range}.  $\Box$

\bigskip

We give now the analogue of Lemma~\ref{l:longpath}.

\begin{lemma}\label{l:longpath2}
There exists a constant $c>0$ such that for any $\ell,k\ge 1$, 
\begin{equation*}
\p\left(\exists |x|\ge \ell\,:\,  N_{y}^{(k)} \ge 2,\, \forall\, y\in   ]\!]  \root,x  [\![ \right) \le 2 k^2{\rm e}^{-c \ell} .
\end{equation*}
\end{lemma}
\noindent {\it Proof}.  The proof is the same as before. The only difference lies in equation (\ref{eq:majlongpath}) where we use the upper bound $\P(N_{y}^{(1)}\ge 1) \le {\rm e}^{-V(y)}$. Then we observe that $\e\left[\sum_{|y|=\ell/2}  {\rm e}^{-2V(x)}\right]$ is exponentially small in $\ell$ by our assumptions. $\Box$ 

\bigskip
\bigskip

\noindent {\it Proof of Theorem~\ref{t:rwre}}. We adapt the proof of Theorem~\ref{t:main}. First we prove that in the annealed and quenched cases, we can restrict to the event $\mathcal E_{n}$ defined in equation (\ref{eq:En})  where we take now $j_{n}:= \fl{\sqrt{n}\ln^5(n)}$.  Here is why we can restrict to the event $\mathcal E_{n}$:  Lemma~\ref{l:range2} and Lemma~\ref{l:longpath2} show that we can restrict to $\{R_{\tau_{\proot}^{(j_n)}} \ge  n,\, \max_{x\in\mathcal Z_{n}} |x|\le \ln^2(n) \}$. Then for any vertex $x\in \t$ and any integer $k$, $\E_{\v}\left[N_{x}^{(k)}\right] = k\ee^{-V(x)}$. The Markov inequality implies that $\P_{\v}(\sum_{|x|\le \ln^2(n)} N_{x}^{(j_n)} > j_n\ln^3(n)) \le {1\over \ln^3(n)}\sum_{k\le \ln^2(n)}\sum_{|x|=k}\ee^{-V(x)}$ which goes to $0$ $\BW$-a.s.\ as $n\to \infty$ since $\sum_{|x|=k} \ee^{-V(x)}$ converges a.s.\ when $k\to\infty$. Then the proof in the annealed case follows the same lines, replacing the measures $\GW$ and $\P_{\t}$ respectively by $\BW$ and $\P_{\v}$. The proof of the quenched case is also similar. The only difference lies in equation (\ref{eq:quenched}). The probability to touch a vertex at generation ${1\over 2} \ln^2(n)$ is smaller than $\max_{|x|={1\over 2}\ln^2(n)} \ee^{-V(x)}$ hence the upper bound in (\ref{eq:quenched}) becomes $nj_{n}  \max_{|x|={1\over 2}\ln^2(n)} \ee^{-V(x)}$. We just need  to show that it is smaller than the general term of a deterministic summable series on an event of $\BW^*$-measure one. This would be implied by the fact that ${\liminf_{k\to\infty}} \min_{|x|=k} {V(x)\over k}>c$ for some constant $c>0$ a.s.,  which holds since  ${1\over \psi(2)^k} \sum_{|x|=k} \ee^{-2V(x)}$  is a nonnegative martingale hence converges a.s.\ and $\psi(2)\in(0,1)$.  $\Box$

\end{document}